\documentclass[11pt,reqno]{amsart}
\usepackage{amsmath, amsthm, amssymb}
\usepackage{mathrsfs}
\usepackage{array}
\usepackage{caption}
\evensidemargin \oddsidemargin
\def\bfz{{\mathbf z}}

\newcommand{\mmod}[1]{\,\,(\text{\rm mod}\,\, #1)}

\def\bfr{{\mathcal{R}}}
\def\bfy{{\boldsymbol y}}
\def\bfx{{\boldsymbol x}}
\def\bfz{{\boldsymbol z}}

\newtheorem{thm}{Theorem}[section]
\newtheorem{cor}{Corollary}[section]

\newtheorem{lem}{Lemma}[section]

\newtheorem{conj}{Conjecture}[section]

\newtheorem{prop}{Proposition}[section]

\numberwithin{equation}{section} \numberwithin{thm}{section}
\numberwithin{lem}{section} \numberwithin{problem}{section}
\numberwithin{cor}{section}

\begin{document}
\title{On the Erd\H{o}s-Tur\'an Conjecture and the growth of $B_{2}[g]$ sequences}
\author[Javier Pliego]{Javier Pliego}
\address{Universit\`a di Genova, Dipartimento di Matematica, Via Dodecaneso 35, 16146 Genova, Italy}
\email{javierpg@kth.se, pliegogarcia@dima.unige.it}

\begin{abstract} When $g\in\mathbb{N}$ we say that $A\subset\mathbb{N}$ is a $B_{2}[g]$ sequence if every $m\in\mathbb{N}$ has at most $g$ distinct representations of the shape $m=b_{1}+b_{2}$ with $b_{1}\leq b_{2}$ and $b_{1},b_{2}\in A$. We show for every $0<\varepsilon<1$ that whenever $g>\frac{1}{\varepsilon}$ then there is a $B_{2}[g]$ sequence $A$ having the property that every sufficiently large $n\in\mathbb{N}$ can be written as $$n=a_{1}+a_{2}+a_{3},\ \ \ \ \ \ \ \ \ a_{3}\leq n^{\varepsilon}\ \ \ \ \ \ \ \ \ a_{i}\in A,$$ and satisfying for large $x$ the estimate $$\lvert A\cap [1,x]\rvert\gg x^{g/(2g+1)}.$$ The above lower bound improves upon earlier results of Cilleruelo and of Erd\H{o}s and Renyi.
\end{abstract}
\maketitle

\section{Introduction}
The problem of representing every sufficiently large integer as a sum of elements of subsets satisfying additional restraints has a long history in the additive theory of numbers, the discussion in the present memoir being confined to the instances in which such subsets have the property that every number has few representations in a prescribed manner involving those. We define an asymptotic basis of order $h$ as a sequence $A\subset\mathbb{N}$ for which every sufficiently large integer $n$ may be written in the shape
$$n=a_{1}+\ldots+a_{h},\ \ \ \ \ \ a_{i}\in A.$$ We also introduce for a sequence $A\subset \mathbb{N}$ and an integer $x\in\mathbb{N}$ the function
$$r_{A}(x)=\Big\{\{a_{1},a_{2}\}:\ \ x=a_{1}+a_{2},\ \ \ \ a_{1},a_{2}\in A\Big\}.$$Equipped with such a utensil we denote by $B_{2}[g]$ to the sequences $A\subset\mathbb{N}$ having the property that $r_{A}(m)\leq g$ for every $m\in\mathbb{N}.$ We shall also refer to $B_{2}[1]$ as Sidon sequences, and focus our attention herein on the investigations concerning the existence of asymptotic bases of order $h$ that are $B_{2}[g]$ sequences simultaneously for suitable choices of $h$ and $g$. 

One may deduce by a simple argument that there are no Sidon asymptotic bases of order two (see \cite{Dirac}). The argument leading to the preceding conclusion is no longer applicable for the problem in which the Sidon property is replaced for arbitrarily large $g\geq 2$ by that of being $B_{2}[g]$, and proving that the aforementioned conditions cannot occur simultaneously  still remains an open question.
\begin{conj}\label{con1}(Erd\H{o}s-Tur\'an conjecture \cite{Erd})
There is no sequence $A\subset\mathbb{N}$ and no fixed $g\geq 2$ with the property that every sufficiently large integer $n$ satisfies
$$1\leq r_{A}(n)\leq g.$$ 
\end{conj}
A significant proportion of the work done concerning such a problem being devoted to shedding some light on how tight the above conjecture is, the end of the present communication is to incorporate to the literature a novel approach to the problem in which we derive the existence of $B_{2}[g]$ sequences satisfying a slightly weaker condition than that of being an asymptotic basis of order $2$. 

\begin{thm}\label{thm11.1}
For each positive integer $g\geq 2$ there exists a $B_{2}[g]$ sequence $A$ having the property that every sufficiently large natural number $n$ can be written as
\begin{equation}\label{refi}n=a_{1}+a_{2}+a_{3},\ \ \ \ \ \ \ a_{3}\ll n^{1/g}(\log n)^{2+1/g},\ \ \ \ \ \ a_{i}\in A.\end{equation} Moreover, the counting function of such a sequence satisfies the lower bound
\begin{equation}\label{infor}\lvert A\cap [1,x]\rvert\gg x^{g/(2g+1)}.\end{equation}
\end{thm}
We record a more transparent version of the preceding theorem in the following corollary.
\begin{cor}\label{thm1.1}
Let $0<\varepsilon<1$. Then for every integer $g>\frac{1}{\varepsilon}$ there exists a $B_{2}[g]$ sequence $A$ having the property that every sufficiently large integer $n$ can be written as
\begin{equation}\label{base3}n=a_{1}+a_{2}+a_{3},\ \ \ \ \ \ \ \ \ a_{3}\leq n^{\varepsilon}\ \ \ \ \ \ \ \ \ a_{i}\in A.\end{equation}  
\end{cor}
Previous results of this flavour would include that of Cilleruelo \cite{Cille3} delivering the existence of a Sidon asymptotic basis of order $4$ with one of the elements in the associated representation being no larger than $n^{\varepsilon}$. We refer the reader to related work \cite{Bou,Woo} in the classical setting of Linnik's conjecture \cite{Lin} on the representability of natural numbers as sums of three squares, one of which being no larger than $n^{\varepsilon}.$ We shall postpone giving account of the literature underlying Conjecture \ref{con1} and observe first that such a conclusion would in turn follow from the stronger conjectural statement (see Erd\H{o}s and Fuchs \cite{ErdFu}) that for any $g\geq 2$ then every $B_{2}[g]$ sequence $A\subset\mathbb{N}$ has the property that 
\begin{equation}\label{addit}\liminf_{x\to\infty}\frac{\lvert A\cap [1,x]\lvert }{x^{1/2}}=0,\end{equation} since asymptotic bases of order $2$ cannot satisfy (\ref{addit}). The above though was proved when $g=1$ by Erd\H{o}s \cite[\S2 Theorem 8]{Hal} by showing that in fact any Sidon sequence $A$ satisfies
\begin{equation*}\liminf_{x\to\infty}\frac{\lvert A\cap [1,x]\rvert(\log x)^{1/2}}{x^{1/2}}=0.\end{equation*}
The existence of a Sidon sequence $A$ satisfying the weaker bound \begin{equation}\label{ASidon}\lvert A\cap [1,x]\rvert\gg x^{1/2-\varepsilon}\end{equation} for arbitrary $\varepsilon>0$ was speculated by Erd\H{o}s \cite{erd3} and is still well beyond our grasp. We find it informative mentioning nonetheless that Ruzsa \cite{Ruz} established the existence of a Sidon sequence $A$ satisfying
$$\lvert A\cap [1,x]\rvert\gg x^{\sqrt{2}-1+o(1)}$$ (see \cite{Cille4} for an explicit construction of a sequence with the above properties). For the instances $g\geq 2$ we record that Erd\H{o}s and Renyi \cite{Erdos} proved in 1960 for arbitrary $\varepsilon>0$ the existence of a $B_{2}[g]$ sequence $A$ having the property for large $x$ that $$\lvert A\cap [1,x]\rvert\gg x^{g/2(g+1)+o(1)}.$$
The preceding bound was superseded by \begin{equation}\label{bounds}\lvert A\cap [1,x]\rvert\gg x^{g/(2g+1)}(\log x)^{-1/(2g+1)+o(1)}\end{equation} in a later memoir of Cilleruelo \cite{Cille2} with the aid of the alteration method and hirtherto remained best possible. We draw the reader's attention to the last part of the statement of Theorem \ref{thm1.1} for the purpose of noting that the lower bound (\ref{infor}) sharpens (\ref{bounds}), such an observation being recorded in the following corollary.
\begin{cor}\label{cor3}For any $g\geq 2$ there is a $B_{2}[g]$ sequence $A\subset\mathbb{N}$ satisfying for large $x$ the estimate
$$\lvert A\cap [1,x]\rvert\gg x^{g/(2g+1)}.$$
\end{cor}
It is indeed a noteworthy feature that our work comprises the first result in the literature that eliminates the undesirable $x^{o(1)}$ factor in the above lower bound. Considerations about the underlying innovations leading to the preceding improvement being deferred, we foreshadow though that the robustness of the concentration inequality that we employ to estimate the upper tail of auxiliary random variables in conjunction with an argument relating the sizes of counting and representation functions precludes one from the necessity of the use of conventional manoeuvres entailing the application of Borel-Cantelli. The absence of such an application permits one to eliminate the logarithmic factor having been required hirthertho for convergence purposes.

We return to Conjecture \ref{con1} and mention that Erd\H{o}s \cite{Erd5}, answering the question posed by Sidon of whether there is an asymptotic basis $A$ of order $2$ for which $r_{A}(n)\ll n^{o(1)}$ further constructed a basis $A$ satisfying $$1\leq r_{A}(n)\ll \log n.$$ A later memoir of Ruzsa \cite{Ru} yields the existence of an asymptotic basis $A$ of order two for which the mean square of the corresponding representation function is bounded, namely $$\sum_{n\leq N}r_{A}(n)^{2}\ll N,$$ the consequence of which being that it is exceptional for a given number $n$ to have a large representation function (see \cite{Ru2}). It has been thought convenient to include a cornucopia of results in the opposite direction, namely the computational work of Grekos et. al. \cite{Gre}, which assures for every asymptotic basis $A$ the lower bound $r_{A}(n)\geq 6$ for infinitely many $n$, the number $6$ having been superseeded by $8$ in \cite{Bor}. The interested reader shall be referred to \cite{Kon} for a result within this circle of ideas.

Back to the setting underlying Theorem \ref{thm1.1}, Erd\H{o}s and Nathanson \cite{erd3} claimed the existence of an asymptotic basis $A$ of order $3$ that is a $B_{2}[g]$ sequence for some $g\in\mathbb{N}$ and asked for the minimum $g$ satisfying such a property. In a later memoir, Cilleruelo \cite{Cille3} demonstrated the validity of the preceding statement for $g=2$, the analogous problem for $g=1$ (see \cite{Erd2,Erd4}) having been solved in a recent preprint of Pilatte \cite{Pil}. We conclude our description of consequences by noting that Theorem \ref{thm1.1} yields when $g=2$ a stronger conclusion than the aforementioned one obtained by Cilleruelo.
\begin{cor}\label{cor3}
There exists a $B_{2}[2]$ sequence $A\subset\mathbb{N}$ having the property that every sufficiently large natural number $n$ can be expressed as
\begin{equation}\label{esste}n=a_{1}+a_{2}+a_{3},\ \ \ \ \ \ \ a_{3}\ll n^{1/2}(\log n)^{5/2},\ \ \ \ \ \ a_{i}\in A.\end{equation} 
\end{cor}
An insigthful inspection of the argument in Pilatte \cite{Pil} reveals that the Sidon sequence $A\subset\mathbb{N}$ stemming from their result satisfies $\lvert A\cap[1,X]\rvert =X^{c+o(1)}$ for $c\leq\frac{3-\sqrt{5}}{2}<2/5$, whence a simple counting argument would entail that such a sequence can not satisfy (\ref{esste}) for all sufficiently large integers $n$. The discussion stemming from Theorem \ref{thm11.1} and equation (\ref{ASidon}) naturally leads to the following prediction.
\begin{conj}
For every $\varepsilon>0$ there is a Sidon sequence $A\subset\mathbb{N}$ with the property that every sufficiently large integer $n$ can be written as
\begin{equation*} n=a_{1}+a_{2}+a_{3},\ \ \ \ \ \ \ \ \ \ a_{3}\leq n^{\varepsilon}\ \ \ \ \ \ \ \ \ a_{i}\in A.\end{equation*} 
\end{conj}

We note in view of Theorem \ref{thm11.1} that any result of the above type with $A$ being instead a $B_{2}[g_{0}]$ sequence for some fixed $g_{0}\geq 2$ would still be interesting (and far out of reach), and record for completeness that Deshoulliers and Plagne \cite{des} had previously constructed a Sidon basis of order 7, the existence of a Sidon basis of order $5$ being delivered in a paper of Kiss \cite{Kis} that was improved by Cilleruelo \cite{Cille3} in the way described after Corollary \ref{thm1.1}.

The proof of Theorem \ref{thm1.1} is based on an application of the probabilistic method introduced by Erd\H{o}s and Renyi \cite{Erdos}, the underlying analysis being deployed in the probability space of sequences for which for $A\subset \mathbb{N}$ then the events $x\in A$ are independent and \begin{equation*}\mathbb{P}(x\in A)\asymp x^{-\alpha}\end{equation*} for some $\alpha>0$. The manoeuvres often utilised to circumvent the dependence of certain random variables by making use of Janson's inequality \cite{Janson1} and which were incorporated by Erd\H{o}s and Tetali \cite{ErTe} are deployed to derive in the setting of Corollary \ref{thm1.1} and for any $$\alpha<\frac{\varepsilon+1}{\varepsilon+2}$$ that
almost every sequence $A$ is an asymptotic basis of order $3$ satisfying (\ref{base3}).

On the other hand, the standard approach entailing the use of the probabilistic method \cite{Erdos} enables one to deduce that almost all sequences are $B_{2}[g]$ whenever $$\alpha>\frac{g+2}{2(g+1)}.$$ The combination of the preceding remarks would enable one to obtain after eliminating a finite number of elements a $B_{2}[g]$ sequence $A$ satisfying (\ref{base3}) and subject to the proviso $g>\frac{2}{\varepsilon}.$ The sharpening of the preceding restraint to both $g>\frac{1}{\varepsilon}$ and \begin{equation*}\alpha>\frac{g+1}{2g+1}\end{equation*} is assessed with the aid of the alteration method in a similar vein as in \cite{Cille3} (see \cite{Cille2} for another application of such a technique) to eliminate from a suitable sequence for which (\ref{base3}) holds the elements involved in the representation of numbers having at least $g+1$ of such representations. However, an astute reappraisal of the argument utilised in \cite{Cille3} to treat the upper tail of the auxiliary random variable $\lvert T_{n}(A)\rvert$ that encodes such elements having its genesis on the Sunflower lemma introduced by Erd\H{o}s and Tetali \cite{ErTe} would have the deficiency that it would essentially require in this setting the assumption that \begin{equation}\label{consr} \mathbb{E}(\lvert T_{n}(A)\rvert)^{2g+4}\ll (\log n)^{2g+4}\ll\mathbb{E}(R_{n}(A)),\end{equation} wherein $R_{n}(A)$ denotes the associated random variable which counts the number of solutions of (\ref{refi}), the concentration inequalities of Vu \cite{Vu2} or Kim and Vu \cite{Kim} utilised on numerous occasions in related problems (see \cite{Vu}) entailing similar obstructions. The preceding restrain would impair the ensuing conclusions and lead to the weaker bound \begin{equation*}\label{bound}\lvert A\cap [1,x]\rvert\gg x^{g/(2g+1)}(\log x)^{-1-2/(2g+1)}\end{equation*}and to (\ref{refi}) but with the restriction $$x_{3}\ll n^{1/g}(\log n)^{4g+\delta}$$ for some constant $\delta>0$. Other concentration inequalities in the literature seem to be applicable in the present setting but do not deliver satisfactory estimates, the main obstruction being that the conditional expectations of some auxiliary random variables that appear in due course do not exhibit in the present setting a power saving $O(n^{-\beta})$ for some $\beta>0$ as is typically required.

Two new ingredients are required, the first one consisting on a robust enough estimate for the upper tail of $\lvert T_{n}(A)\rvert$ which hinges on the analysis of integral moments of $\lvert T_{n}(A)\rvert$ and circumvents the aforementioned power saving constraints. Such an alleviation permits one to eliminate the exponent in (\ref{consr}) and the logarithmic factor in (\ref{infor}). The main difficulty then has its genesis on the high dependence of some of the random variables involved when expanding the corresponding product, it being ultimately required showing after a careful division of cases that such contributions are in most of the instances negligible. The sequences stemming after such considerations have the additional property that the number of representations of $n$ as in (\ref{refi}) is large enough so as to derive the required lower bound for the corresponding counting function, such a manoeuvre thereby avoiding the use of Borel-Cantelli. 

The exposition is structured as follows. After giving account of required technical lemmata, Sections \ref{sec3} and \ref{sec4} are devoted to deduce that almost all sequences have the property that (\ref{refi}) holds. In Section \ref{sec5} we outline the argument to assure that such sequences satisfy the additional $B_{2}[g]$ condition, the analysis of both the expectation and the upper tail of $\lvert T_{n}(A)\rvert$ being required to derive such a condition and assessed in Sections \ref{sec7}, \ref{sec10}, \ref{sec11} and \ref{sec9}. Section \ref{sec8} includes then the corresponding discussion about the lower bound for the size of the counting function underlying Corollary \ref{cor3} and the proof of Theorem \ref{thm1.1}. We use $\ll$ and $\gg$ to denote Vinogradov's notation, write $f=O(g)$ if $f\ll g$ and $f\asymp g$ whenever $f\ll g$ and $f\gg g$. If the latter symbol is involved to describe a range of summation, the implicit constants will not play a role in the argument. When $\lambda>0$ we also write $\ll_{\lambda}$ and $O_{\lambda}(-)$ to specify that the corresponding implicit constants depend on $\lambda$. We shall use $g=o_{\lambda}(f)$ to indicate that $\lim_{\lambda\rightarrow\infty}f/g=0.$ We write $\bfx\in\mathbb{R}^{k}$ to denote vectors $\bfx=(x_{1},\ldots,x_{k})$.

\emph{Acknowledgements:}
The author was supported by the G\"oran Gustafsson Foundation while being a Postdoc at KTH, would like to thank Javier Cilleruelo for introducing him to the topic of the paper and express his gratitude to Akshat Mudgal for inspiring remarks.

\section{Preliminary lemmata}
We begin by introducing some required infraestructure, it being convenient to present first the following lemma that establishes the existence of the probability space on which to base our manoeuvres. We write $\Omega$ to refer to the set of sequences of the natural numbers.
\begin{lem}\label{proba}
Let $(\theta_{n})_{n}$ be a sequence of real numbers having the property that
$$0\leq \theta_{n}\leq 1\ \ \ \ \ \ \ \ \ \ n\in\mathbb{N}.$$ Then there exists a probability space $(\Omega, \mathcal{A}, \mathbb{P})$ satisfying:

(i) For every $n\in\mathbb{N}$, the event $\mathcal{B}^{(n)}=\{\omega\subset \mathbb{N}: \ \ n\in\omega\}$ is measurable and $\mathbb{P}(\mathcal{B}^{(n)})=\theta_{n}.$

(ii) The events $\mathcal{B}^{(1)},\mathcal{B}^{(2)},\ldots$ are independent.
\end{lem}
\begin{proof}
See Halberstam-Roth \cite[Theorem 13, \S 3]{Hal}.
\end{proof}
The majority of the results obtained herein which shall eventually lead to establishing a particular conclusion in a set of measure $1$ and for a sufficiently large integer shall have its reliance on the Borel-Cantelli lemma.
\begin{lem}[Borel--Cantelli] \label{prop2} Let $\{E_{n}\}_{n=1}^{\infty} $ be a sequence of measurable events having the property that $$\sum_{n=1}^{\infty}\mathbb{P}(E_{n})<\infty.$$ Then with probability $1$ at most a finite number of them occur. Equivalently, one has that
$$\mathbb{P}\Big(\bigcap_{i=1}^{\infty}\bigcup_{j=i}^{\infty}E_{j}\Big)=0.$$
\end{lem}
\begin{proof}
See Halberstam-Roth \cite[Theorem 7, \S 3]{Hal}.
\end{proof}

The upcoming technical lemma shall be utilised on numerous occasions in the memoir, it being profitable presenting it herein and alluding to its conclusion when required henceforth. 

\begin{lem}\label{lem1.1}
	For all $\alpha,\beta>0$ with $1/2<\alpha,\beta<1$ we have
	$$ \sum_{\substack{x,y\geq 1\\  x+y=n}}x^{-\alpha}y^{-\beta}\ll n^{1-\alpha-\beta},\ \ \ \ \ \ \ \ \sum_{\substack{x,y\geq 1\\  y-x=n}}{x^{-\alpha}y^{-\beta}}\ll n^{1-\alpha-\beta}.$$ 
\end{lem}

\begin{proof}
We denote by $S_{1}$ to the first sum and note by following a routinary procedure that
\begin{align*}
S_{1}=&\sum_{1 \leq x \leq \frac{n}{2}}x^{-\alpha}(n-x)^{-\beta}+\sum_{\frac{n}{2}< x \leq n}x^{-\alpha}(n-x)^{-\beta},
\end{align*} whence it is apparent whenever $\alpha, \beta<1$ that 
\begin{align*}S_{1}&\ll n^{-\beta}\sum_{1 \leq x \leq \frac{n}{2}}x^{-\alpha}+n^{-\alpha}\sum_{ 1\leq x \leq n/2}x^{-\beta}\ll n^{1-\alpha-\beta}.
\end{align*}
Likewise, by employing an analogous argument it transpires that
\begin{align*}
\sum_{\substack{x,y\geq 1\\  y-x=n}}{x^{-\alpha}y^{-\beta}}=&\sum_{1 \leq x \leq n}x^{-\alpha}(n+x)^{-\beta}+\sum_{x>n}x^{-\alpha}(n+x)^{-\beta}
\\
\ll &n^{-\beta}\sum_{1 \leq x \leq n}x^{-\alpha}+\sum_{ x>n}x^{-\alpha-\beta}\ll n^{1-\alpha-\beta},
\end{align*}
as desired. 
\end{proof}
\section{Expected value of $R_{n}(A)$}\label{sec3}
In order to prepare the ground for the application of the probabilistic method, it seems worth introducing first the following theorem to alleviate further computations.
\begin{thm}\label{thm2.1}
There exists infinitely many cyclic groups $\mathbb{Z}/M\mathbb{Z}$ containing a Sidon set $S\subset \mathbb{Z}/M\mathbb{Z}$ and having the property for every $n\in \mathbb{Z}/M\mathbb{Z}$ that \begin{equation*}n=s_{1}+s_{2}+s_{3},\ \ \ \ \ \ \ s_{i}\in S\end{equation*} for pairwise distinct $s_{1},s_{2},s_{3}\in S$.
\end{thm}
\begin{proof}
See Cilleruelo \cite[Theorem 2.1]{Cille3}.
\end{proof}
Equipped with the preceding result we then fix a sufficiently large constant $M>0$ and a Sidon set $C_{M}\subset\mathbb{Z}/M\mathbb{Z}$ satisfying the conclusion of the above theorem and introduce for fixed integer $g\geq 2$ and the parameters \begin{equation}\label{alpha}\alpha=\frac{g+1}{2g+1},\ \ \ \ \ \ \ \ \ \ \ \ \ \ \  \varepsilon=\frac{1}{g},\end{equation} and for a sufficiently large constant $\lambda=\lambda(g)$ the probability space $\mathcal{S}(\alpha,\lambda,C_{M})$ stemming from Lemma \ref{proba} with

\begin{equation}\label{probas}\mathbb{P}(x\in A)=\left\{
	       \begin{array}{ll}
	\lambda^{-1} x^{-\alpha}\ \ \ \ \ \ \ \ \ \ \ \ \ \ \ \ \ \ \ \ \ \ x\equiv c\mmod{M}\text{ for some $c\in C_{M}$}   \\

		0 \ \ \ \ \ \ \ \ \ \ \ \ \ \ \ \ \ \ \ \ \ \ \ \ \ \ \ \ \ \ \text{otherwise.}               \\
	       \end{array}
	     \right.
 \end{equation}We shall also denote for $\bfx=(x_{1},\ldots,x_{h})$ and further convenience $\text{Set}(\bfx)=\{x_{1},\ldots,x_{h}\},$ and introduce when $n\in\mathbb{N}$ the auxiliary function \begin{equation}\label{feo}g_{\lambda}(n)=\lambda^{\frac{2g+5}{1-\alpha}} n^{\varepsilon}(\log n)^{\frac{1}{1-\alpha}},  \end{equation} the box $\mathcal{B}_{n}=[g_{\lambda}(n)/2,g_{\lambda}(n)]\times [n/4,n]^{2}$ and the set 
\begin{equation}\label{prras}R_{n}=\Big\{\{x_{1},x_{2},x_{3}\}\in\mathbb{N}: \ n=x_{1}+x_{2}+x_{3},\ \ \ \ (x_{1},x_{2},x_{3})\in \mathcal{B}_{n}, \ \ \ \ x_{i}\equiv s_{i}\mmod{M}\Big\},\end{equation} wherein $s_{i}$ satisfy the conclusion of Theorem \ref{thm2.1}, and define for each $A\subset\mathbb{N}$ the representation function
$$ R_{n}(A)=\lvert\{\omega \in R_{n}: \omega \subset A\}\rvert.$$ 

It seems worth observing that one may write 
$$R_{n}(A)=\sum_{\substack{\{x_{1},x_{2},x_{3}\}\in R_{n}}}\leavevmode\hbox{$1\!\rm I$}_{A}(x_{1},x_{2},x_{3}),$$
where $\leavevmode\hbox{$1\!\rm I$}_{A}(x_{1},x_{2},x_{3})$ equals 1 if $x_{1},x_{2},x_{3}\in A$ and $0$ else. It transpires then that the above expression is a random variable, the computation of its expected value being presented shortly. We note in view of the definition of $R_{n}$ that the coordinates of the underlying tuples $(x_{1},x_{2},x_{3})$ belong to distinct congruence classes and hence are different, the events $$x_{1}\in A,\ \ \ \ \ \ x_{2}\in A,\ \ \ \ \ x_{3}\in A$$ being therefore independent, whence applying the linearity of the expectation yields
\begin{align*}
\mathbb{E}(R_{n}(A))=&\sum_{\{x_{1},x_{2},x_{3}\}\in R_{n}}\mathbb{P}(x_{1},x_{2},x_{3}\in A)=\sum_{\{x_{1},x_{2},x_{3}\}\in R_{n}}\mathbb{P}(x_{1}\in A)\mathbb{P}(x_{2}\in A)\mathbb{P}(x_{3}\in A).
\end{align*}
We record for further use that with the preceding definitions one has
\begin{equation}\label{varita}\varepsilon(1-\alpha)+1-2\alpha=0\end{equation} and examine the asymptotic behaviour of the above expectation.
\begin{lem}\label{lem1.2}
One has for sufficiently large $n$ that \begin{equation*}\mathbb{E}(R_{n}(A))\asymp \lambda^{2g+2}\log n.\end{equation*}
\end{lem}
\begin{proof}
We apply the definition (\ref{probas}) and (\ref{varita}) to the preceding line and obtain
\begin{align}\label{ela}\mathbb{E}(R_{n}(A))\ll \lambda^{-3}\sum_{\substack{  \{x_{1},x_{2},x_{3}\}\in R_{n}}}(x_{1}x_{2}x_{3})^{-\alpha}&\ll\lambda^{-3}g_{\lambda}(n)^{-\alpha}n^{-2\alpha}\lvert R_{n}\rvert\ll \lambda^{2g+2}\log n.\end{align}
Likewise, we observe that
$$\mathbb{E}( R_{n}(A))\gg \lambda^{-3-\alpha\frac{2g+5}{1-\alpha}}n^{-\alpha(2+\varepsilon)}(\log n)^{-\frac{\alpha}{1-\alpha}}\lvert R_{n}\rvert.$$
We also note that $\lvert R_{n}\rvert$ is equal to the number of triples $(y_{1},y_{2},y_{3})\in\mathbb{N}^{3}$ having the property that $$y_{1}+y_{2}+y_{3}=(n-s_{1}-s_{2}-s_{3})/M$$ and satisfying the additional constraints $y_{1}\in [\big(g_{\lambda}(n)/2-s_{1}\big)/M,\big(g_{\lambda}(n)-s_{1}\big)/M]$ and $$y_{i}\in [(n/4-s_{i})/M,(n-s_{i})/M],\ \ \ \ \ \ \ i=2,3,$$ whence the preceding discussion enables one to deduce the estimate
$$\lvert R_{n}\rvert\asymp \lambda^{\frac{2g+5}{1-\alpha}} n^{1+\varepsilon}(\log n)^{\frac{1}{1-\alpha}}M^{-2}.$$ We remind the reader of the fact that the parameter $M$ is fixed, such an observation in conjunction with the above lower bound and (\ref{varita}) permiting one to derive 
$$\mathbb{E}( R_{n}(A))\gg \lambda^{2g+2}\log n,$$ as desired.

\end{proof}
\section{The lower tail}\label{sec4}
We shall make use of Janson's inequality \cite{Janson1} in the present section to provide a lower bound for the random variable $R_{n}(A)$ in a set with high probability, it being pertinent to such an end to verify some preliminary required conditions. 
\begin{prop}[Janson's inequality]\label{prop1}   Let $\Omega$ be a family of sets and let $A$ be a random subset. We further consider $X(A)=|\{\omega \in \Omega: \omega \subset A\}|$ with finite expected value $\mu=\mathbb{E}(X(A))$. Then for fixed $\delta>0$ one has
$$\mathbb{P}(X\leq(1-\delta)\mu)\leq \exp(-\delta^2\mu^2/(2\mu+\Delta(\Omega))),
$$
where
\begin{equation}\label{omegita}\Delta(\Omega)=\sum_{\substack{\omega,\omega'\in\Omega\\ \omega\sim\omega'}}\mathbb{P}(\omega,\omega'\subset A)\end{equation}
and\ $\omega\sim\omega'$ denotes $\omega\ \cap\omega'\neq \o$ with $\omega\neq\omega'$. In particular, upon making the choice $\delta= {1}/{2}$, one has when $\Delta(\Omega)<\mu$ that
$$\mathbb{P}(X\leq \mu/2)\leq \exp(-\mu/12).$$
\end{prop}
\begin{proof}
See \cite{Janson1}.
\end{proof}
In view of the above considerations and upon taking $X(A)=R_{n}(A)$ in the context underlying Janson's inequality, the preceding discussion reduces the problem to that of obtaining the estimate $\Delta(R_{n})<\mu_{n}$ for sufficiently large $n$, wherein $\mu_{n}=\mathbb{E}(R_{n}(A))$.

\begin{lem}\label{lem1.3}
Whenever $n\rightarrow \infty $ one has that \begin{equation}\label{pasta}\Delta\big(R_{n}\big)=o (\mu_{n}).\end{equation}
\end{lem}

\begin{proof} 
It is apparent that
\begin{equation*}\Delta(R_{n})\ll \sum_{\substack{\{x_{1},x_{2},x_{3}\} \in R_{n}\\\{x_{1},x'_{2},x'_{3}\}\in R_{n}\\ \{x_{2},x_{3}\}\neq \{x_{2}',x_{3}'\}}}\mathbb{P}(x_{1},x_{2},x_{3},x'_{2},x'_{3} \in A),\end{equation*}
wherein we employed the proviso $\omega\neq \omega'$ underlying the corresponding sum in (\ref{omegita}) in conjunction with the fact that the cognate elements of the sets satisfy the lineal equation $x_{1}+x_{2}+x_{3}=n$ to deduce that $\lvert\omega\cap \omega'\rvert=1.$ We recall to the reader that all the variables involved are pairwise disjoint and distinguish among both the instance when the underlying intersection of the sets comprises an element $x_{1}\leq g_{\lambda}(n)$ and when it does not, namely
\begin{align*} \nonumber
\Delta(R_{n})&\ll_{\lambda}\sum_{x_{1}\asymp g_{\lambda}(n)}x_{1}^{-\alpha}\bigg(\sum_{\substack{\{x_{1},x_{2},x_{3}\}\in R_{n}}}(x_{2}x_{3})^{-\alpha}\bigg)^2+\sum_{\substack{x_{2}\asymp n}}x_{2}^{-\alpha}\bigg(\sum_{\substack{x_{1}+x_{3}=n-x_{2}\\ x_{1}\asymp g_{\lambda}(n)\\ x_{3}\asymp n}}x_{1}^{-\alpha}x_{3}^{-\alpha}\bigg)^2
\\
&\ll_{\lambda} n^{2-4\alpha}\sum_{x_{1}\asymp g_{\lambda}(n)}x_{1}^{-\alpha}+(\log n)^{B}n^{-2\alpha+2\varepsilon(1-\alpha)}\sum_{\substack{x_{2}\asymp n}}x_{2}^{-\alpha}
\end{align*}
for some big enough constant $B>0$, which then enables one to deduce
\begin{equation}\label{prisa}\Delta(R_{n})\ll_{\lambda} (\log n)^{B}(n^{\varepsilon(1-\alpha)+2-4\alpha}+n^{1-3\alpha+2\varepsilon(1-\alpha)}).\end{equation}

We shall omit indicating that $B>0$ is some constant which may vary from line to line throughout the memoir for the sake of brevity. It seems worth noting upon recalling (\ref{varita}), the proviso $\varepsilon<1$ and the inequality $1/2<\alpha<1$, these in turn stemming from (\ref{alpha}), that  
\begin{equation*}1-3\alpha+2\varepsilon(1-\alpha)<2(1-2\alpha)+\varepsilon(1-\alpha)<\varepsilon(1-\alpha)+1-2\alpha=0.
\end{equation*}
Therefore, combining equation (\ref{prisa}) with the above lines permits one to deduce 
$$\Delta(R_{n})=o(n^{-\delta})$$ for some fixed $\delta>0$, it entailing (\ref{pasta}) when applied in conjunction with Lemma \ref{lem1.2}.

\end{proof}
In view of the preceding conclusion, it transpires that we have reached a position from which to apply Janson's inequality.
\begin{lem}\label{lem1.4}
For a sufficiently large constant $\lambda=\lambda(g)$ one has with probability $1$ that $$R_{n}(A)\gg \lambda^{2g+2}\log n$$ holds for $n$ sufficiently large.
\end{lem}
\begin{proof}
We observe first that an application of Lemma \ref{lem1.3} in conjunction with Lemma \ref{lem1.2} and Proposition \ref{prop1} enables one to deduce the estimate 
\begin{align*}\mathbb{P}(R_{n}(A)\leq \mu_{n}/2)\leq \exp(-\mu_{n}/12)&\ll  \exp\big(-c\lambda^{2g+2}(\log n)\big)
\end{align*} for some constant $c>0$. The reader may find it useful to observe when $\lambda$ is a sufficiently large constant that the right side of the above line is $O(n^{-2})$. Consequently,
\begin{equation*}
\sum_{n=1}^{\infty}\mathbb{P}\big(R_{n}(A)<\mu_{n}/2\big)<\infty.
\end{equation*}
The statement follows combining the above line with a routine application of Lemma \ref{prop2} (Borel-Cantelli).
\end{proof}

\section{Construction of $B_{2}[g]$ sequences}\label{sec5}
Sketching the argument that shall eventually enable one to effectively control the representation function $r_{A}(x)$ is the main purpose of the present section. The approach taken henceforth will consist of a procedure which assures that with high probability the contribution to $R_{n}(A)$ of triples having the property that at least one of their components is involved in the representation of some integer having at least $g+1$ representations in the prescribed manner is of a significantly smaller size than $R_{n}(A)$. We shall put ideas into effect by recalling (\ref{prras}) and furnishing ourselves for any $g\geq 2$ with the definition of the set 
\begin{equation}\label{jood}T_{n}=\Big\{\bfx\in\mathbb{N}^{2g+4}:\ \{x_{1},x_{2},x_{3}\}\in R_{n},\ \ x_{1}+x_{4}=x_{5}+x_{6}=\ldots=x_{2g+3}+x_{2g+4},\ \ x_{1}\geq x_{i}\Big\},\end{equation} wherein the variables satisfy the additional restraints 
\begin{equation}\label{modu}x_{1}\equiv x_{2k+1}\mmod{M}\ \ \ \text{for $2\leq k\leq g+1$},\ \ \ \ \ \ \ \ \ \ \ \ x_{4}\equiv x_{2k+2}\mmod{M}\ \ \ \text{for $2\leq k\leq g+1$}\end{equation} and
\begin{equation}\label{kkk}\{x_{1},x_{4}\}\neq \{x_{2k+1},x_{2k+2}\}\neq \{x_{2j+1},x_{2j+2}\}\end{equation} for every $2\leq j,k\leq g+1$ with $j\neq k$. It also seems worth considering the subset $$T_{n}^{ind}=\Big\{\bfx\in T_{n}:\ \ \ x_{i}\neq x_{j}\ \ \ \text{for $1\leq i< j\leq 2g+4$}\Big\}.$$
Moreover, we introduce for a given sequence of integers $A\subset\mathbb{N}$ the sets
\begin{equation}\label{AAAAA}T_{n}(A)=\Big\{\bfx\in T_{n}:\ \ \ \text{Set}(\bfx)\subset A\Big\},\ \ \ \ \ \ \ T_{n}^{ind}(A)=\Big\{\bfx\in T_{n}^{ind}:\ \ \ \text{Set}(\bfx)\subset A\Big\}.\end{equation} 

Equipped with the above considerations, we construct for each $A\subset\mathbb{N}$ a sequence of integers $A'$ by deleting from $A$ every $a_{1}\in A$ having the property that
$$a_{1}+a_{2}=a_{3}+a_{4}=\ldots=a_{2g+1}+a_{2g+2},\ \ \ \ \ \ \ \ \ a_{1}\geq a_{i},$$ for some $a_{i}\in A$ and such that $$\{a_{2i-1},a_{2i}\}\neq \{a_{2j-1},a_{2j}\} \ \ \ \ \ \ \text{for $1\leq i<j\leq g+1$}$$
with $a_{i}\equiv c_{i}\mmod{M}$ for some $c_{i}\in C_{M}$, the latter being the Sidon set $C_{M}$ employed in (\ref{probas}). In view of the preceding remarks, it transpires upon recalling (\ref{prras}) to the reader that any set $\{x_{1},x_{2},x_{3}\}\subset A$ counted by $R_{n}(A)$ with one of the elements, say $x_{1}$, being removed in the construction of $A'$ satisfies in particular $\{x_{1},x_{2},x_{3}\}\in R_{n}$ with
$$x_{1}+x_{4}=x_{5}+x_{6}=\ldots=x_{2g+3}+x_{2g+4},\ \ \ \ \ \ x_{1}\geq x_{i}\ \ \ \ \ \text{for $x_{i}\in A,\ \ 4\leq i\leq 2g+4$,}$$ having the customary property that \begin{equation*}\{x_{1},x_{4}\}\neq \{x_{2i+1},x_{2i+2}\}\neq \{x_{2j+1},x_{2j+2}\} \ \ \ \ \ \ \text{for $2\leq i,j\leq g+1$ with $i\neq j$}.\end{equation*} Such properties and the fact that $C_{M}$ is Sidon then enables one to deduce by relabelling if required that the variables satisfy (\ref{modu}), an ensuing consequence of which being that 
$$0\leq R_{n}(A)-R_{n}(A')\leq \lvert T_{n}(A)\rvert,$$ it in turn entailing the validity of the expression
\begin{equation}\label{AAA}R_{n}(A')= R_{n}(A)+O\big( \lvert T_{n}(A)\rvert\big).\end{equation} Since $A'$ is a $B_{2}[g]$ sequence, the above assertion in conjunction with Lemma \ref{lem1.4} thereby reduces our assessment to that of establishing with high probability that $\lvert T_{n}(A)\rvert=o_{\lambda}(R_{n}(A))$.

\section{Expected value of $\lvert T_{n}(A)\rvert$.}\label{sec7}
We shall proceed in the upcoming sections to provide an upper bound for $\lvert T_{n}(A)\rvert$ of the requisite precision almost surely. To such an end, we introduce first the sets
\begin{equation}\label{alej}T_{n}^{dep}=\Big\{\bfx\in T_{n}:\ \ \ \ x_{i}=x_{j}\ \text{for some $i\neq j$}\Big\},\ \  \ T_{n}^{dep}(A)=\Big\{\bfx\in T_{n}^{dep}:\ \ \ \text{Set}(\bfx)\subset A\Big\}\end{equation}
and furnish ourselves with the computation of their expected values as a prelude to the discussion concerning the upper tail of $\lvert T_{n}(A)\rvert$.
\begin{lem}\label{lem7}
For sufficiently large $n$ and fixed constant $\lambda>1$ one has
\begin{equation}\label{als}\mathbb{E}(\lvert T_{n}^{ind}(A)\rvert)\ll \lambda \log n,\ \ \ \ \ \ \ \ \ \ \ \ \mathbb{E}(\lvert T_{n}^{dep}(A)\rvert)\ll_{\lambda} n^{(1-2\alpha)\varepsilon}\log n.
\end{equation} 

\end{lem}
\begin{proof}

We start our examination by analysing first $\mathbb{E}(\lvert T_{n}^{ind}(A)\rvert)$ and obtain via subsequent applications of Lemma \ref{lem1.1} that
\begin{align*}
\mathbb{E}(\lvert T_{n}^{ind}(A)\rvert)&\ll \lambda^{-(2g+4)}\sum_{\substack{\{x_{1},x_{2},x_{3}\}\in R_{n}}}(x_{1}x_{2}x_{3})^{-\alpha}\sum_{x_{4}\leq x_{1}}x_{4}^{-\alpha}\Big(\sum_{x+y=x_{1}+x_{4}}(xy)^{-\alpha}\Big)^{g}
\\
&\ll \lambda^{-(2g+4)}\sum_{\substack{\{x_{1},x_{2},x_{3}\}\in R_{n}}}(x_{1}x_{2}x_{3})^{-\alpha}\sum_{x_{4}\leq x_{1}}x_{4}^{-\alpha}(x_{1}+x_{4})^{(1-2\alpha)g}
\\
&\ll \lambda^{-(2g+4)}\sum_{\substack{\{x_{1},x_{2},x_{3}\}\in R_{n}}}(x_{2}x_{3})^{-\alpha}x_{1}^{(1-2\alpha)(g+1)},
\end{align*}
where in the third step we used the constraint $x_{4}\leq x_{1}$. It seems opportune to observe in view of (\ref{alpha}) that
\begin{equation}\label{esti} (1-2\alpha)(g+1)= -\alpha,\ \ \ \ \ \ \ \ \ \ \ \ g+1-(2g+1)\alpha=0,\end{equation} 
whence the preceding discussion then permits one to deduce that
$$\mathbb{E}(\lvert T_{n}^{ind}(A)\rvert)\ll \lambda^{-(2g+4)}\sum_{\substack{ \{x_{1},x_{2},x_{3}\}\in R_{n}}}(x_{1}x_{2}x_{3})^{-\alpha}\ll \lambda\log n,$$ wherein we employed (\ref{ela}).

We foreshadow as above that the underlying congruence conditions satisfied by the appertaining variables alleviate the computations required to derive the second bound in (\ref{als}), it being worth introducing first the sets
\begin{equation*}\label{joo}T_{n,1}=\Big\{\bfx\in\mathbb{N}^{2g+3}:\ \{x_{1},x_{2},x_{3}\}\in R_{n},\ \ x_{1}+x_{2}=x_{5}+x_{6}=\ldots=x_{2g+3}+x_{2g+4}\Big\},\end{equation*} 
\begin{equation*}\label{joo}T_{n,2}=\Big\{\bfx\in\mathbb{N}^{2g+3}: \{x_{1},x_{2},x_{3}\}\in R_{n},\ x_{1}+x_{4}=x_{5}+x_{2}=\ldots=x_{2g+3}+x_{2g+4},\ x_{2},x_{4}\leq x_{1}\Big\},\end{equation*} 
\begin{equation*}\label{joo}T_{n,3}=\Big\{\bfx\in\mathbb{N}^{2g+3}:\ \{x_{1},x_{2},x_{3}\}\in R_{n},\ \ x_{1}+x_{4}=x_{5}+x_{6}=\ldots=2x_{2g+3},\ \ x_{4}\leq x_{1}\Big\},\end{equation*} 
\begin{equation*}\label{joo}T_{n,4}=\Big\{\bfx\in\mathbb{N}^{2g+3}:\ \{x_{1},x_{2},x_{3}\}\in R_{n},\ \ 2x_{1}=x_{5}+x_{6}=\ldots=x_{2g+3}+x_{2g+4}\Big\}\end{equation*} and to define
$$T_{n,i}(A)= \Big\{\bfx\in T_{n,i}:\ \ \ \text{Set}(\bfx)\subset A\Big\}$$ for a given sequence $A\subset \mathbb{N}$. We split then the corresponding sum into cases, namely
 $$\mathbb{E}(\lvert T_{n}^{dep}(A)\rvert)\ll \sum_{i=1}^{4}\mathbb{E}(\lvert T_{n,i}(A)\rvert).$$

It may be opportune to clarify that the summand corresponding to $i=2$ gives account of the instances in which $\lvert\{x_{2},x_{3}\}\cap \big\{x_{2j},\ \ \ 3\leq j\leq g+2\big\}\rvert=1$, the cases $x_{2}=x_{2j}$ for $j>3$ or $x_{3}=x_{2j}$ for $j\geq 3$ being completely symmetrical to that presented in the pertaining sum. We have also considered the cases in which either $x_{2}=x_{4}$ or $x_{3}=x_{4}$, the latter instances having been encompassed in the term $i=1$. The summand corresponding to $i=3$ comprises the cases $x_{2j+1}=x_{2j+2}$ for $j\geq 2,$ the instance $x_{1}=x_{4}$ being considered in the summand $i=4$. An insightful inspection of the description of $T_{n}$ in (\ref{AAAAA}) with special attention to both (\ref{kkk}) and the underlying congruence relations (\ref{modu}) reveals that the above cases constitute the only instances in which two or more of the cognate variables are equal, for if \begin{equation}\label{gui}\lvert\{x_{2},x_{3}\}\cap \big\{x_{2i+1},\ \ \ 2\leq i\leq g+1\big\}\rvert=1,\end{equation} say $x_{2}=x_{2i_{0}+1}$, then $x_{1}\equiv x_{2i_{0}+1} \mmod{M}$ but $x_{1}\not\equiv x_{2}\mmod{M}$, which is a contradiction. By a similar reason, if $x_{2}=x_{2i}$ and $x_{3}=x_{2j}$ for $i\neq j$ then $x_{2}\equiv x_{2i}\equiv x_{2j}\equiv x_{3} \mmod{M}$, which would again contradict our assumptions.

We shift our attention to the term $T_{n,1}$ and proceed by applying Lemma \ref{lem1.1} to get
\begin{equation*}\mathbb{E}(\lvert T_{n,1}(A)\rvert)\ll \sum_{\substack{\{x_{1},x_{2},x_{3}\}\in R_{n}}}(x_{1}x_{2}x_{3})^{-\alpha}(x_{1}+x_{2})^{(1-2\alpha)g}\ll_{\lambda}n^{(1-2\alpha)g}\log n,\end{equation*} where in the last step we utilised the fact that $x_{1}+x_{2}\gg n$ and (\ref{ela}). For the term $T_{n,2}$ we observe as is customary that
$$\mathbb{E}(\lvert T_{n,2}(A)\rvert)\ll \sum_{\substack{\{x_{1},x_{2},x_{3}\}\in R_{n}\\x_{1}>x_{2}}}(x_{1}x_{2}x_{3})^{-\alpha}\sum_{\substack{x_{4}\leq x_{1}\\ x_{1}+x_{4}=x_{5}+x_{2}}}(x_{4}x_{5})^{-\alpha}\Bigg(\sum_{x+y=x_{1}+x_{4}}(xy)^{-\alpha}\Bigg)^{g-1}.$$ An application of Lemma \ref{lem1.1} then delivers
\begin{align*}\mathbb{E}(\lvert T_{n,2}(A)\rvert)&\ll\sum_{\substack{ \{x_{1},x_{2},x_{3}\}\in R_{n}\\ x_{1}>x_{2}}}(x_{1}x_{2}x_{3})^{-\alpha}\sum_{\substack{x_{4}\leq x_{1}}}x_{4}^{-\alpha}(x_{4}+x_{1}-x_{2})^{-\alpha}(x_{1}+x_{4})^{(1-2\alpha)(g-1)}
\\
&\ll \sum_{\substack{\{x_{1},x_{2},x_{3}\}\in R_{n}\\ x_{1}>x_{2}}}(x_{2}x_{3})^{-\alpha}(x_{1}-x_{2})^{-\alpha}x_{1}^{(1-2\alpha)g}.
\end{align*}
We observe in view of the condition $x_{1}>x_{2}$ that $x_{1}\asymp n$ and note by (\ref{varita}) that then
\begin{align*}\mathbb{E}(\lvert T_{n,2}(A)\rvert)&\ll n^{g-(2g+2+\varepsilon)\alpha}\sum_{\substack{\{x_{1},x_{2},x_{3}\}\in R_{n}\\ x_{2}\asymp g_{\lambda}(n)}}1+n^{g-(2g+1+\varepsilon)\alpha}\sum_{\substack{x_{3}\asymp g_{\lambda(n)}\\ 2x_{2}<n-x_{3}}}(n-x_{3}-2x_{2})^{-\alpha}
\\
&\ll (\log n)^{B}n^{(g+1)(1-2\alpha)+\varepsilon(1-\alpha)}\ll n^{1-2\alpha},
\end{align*}
as desired. We shall focus our attention next on the perusal of $T_{n,3}$ and note that it follows in a routinary manner that
\begin{align*}\mathbb{E}(\lvert T_{n,3}(A)\rvert)&\ll \sum_{\substack{\{x_{1},x_{2},x_{3}\}\in R_{n}}}(x_{1}x_{2}x_{3})^{-\alpha}\sum_{\substack{x_{1}+x_{4}=2x_{2g+3}\\ x_{4}\leq x_{1}}}(x_{4}x_{2g+3})^{-\alpha}(x_{1}+x_{4})^{(1-2\alpha)(g-1)}
\\
&\ll \sum_{\substack{\{x_{1},x_{2},x_{3}\}\in R_{n}}}(x_{1}x_{2}x_{3})^{-\alpha}\sum_{x_{4}\leq x_{1}}x_{4}^{-\alpha}(x_{1}+x_{4})^{(1-2\alpha)(g-1)-\alpha}
\\
&\ll \sum_{\substack{\{x_{1},x_{2},x_{3}\}\in R_{n}}}x_{1}^{(1-2\alpha)g-\alpha}(x_{2}x_{3})^{-\alpha}\ll n^{(1-2\alpha)g\varepsilon}\sum_{\substack{\{x_{1},x_{2},x_{3}\}\in R_{n}}}(x_{1}x_{2}x_{3})^{-\alpha},
\end{align*} whence it transpires after an application of (\ref{ela}) that then \begin{equation*}\mathbb{E}(\lvert T_{n,3}(A)\rvert)\ll_{\lambda} n^{(1-2\alpha)g\varepsilon}(\log n).\end{equation*}
We shall conclude the proof by employing as is customary Lemma \ref{lem1.1} to derive
\begin{align*}\mathbb{E}(\lvert T_{n,4}(A)\rvert)&\ll \sum_{\substack{\{x_{1},x_{2},x_{3}\}\in R_{n}}}(x_{1}x_{2}x_{3})^{-\alpha}\Bigg(\sum_{x+y=2x_{1}}(xy)^{-\alpha}\Bigg)^{g}\nonumber
\\
&\ll \sum_{\substack{\{x_{1},x_{2},x_{3}\}\in R_{n}}}(x_{2}x_{3})^{-\alpha}x_{1}^{(1-2\alpha)g-\alpha}\ll_{\lambda} n^{(1-2\alpha)g\varepsilon}(\log n),
\end{align*}
the last estimate stemming from the above treatment. The preceding bounds and remarks then enable one to complete the proof of the lemma.
\end{proof}

\section{Auxiliary estimates for conditional expectations}\label{sec10}
The present section shall be devoted to provide a prolix analysis examining some of the conditional expectations that arise in the discussion concerning the upper tail of $\lvert T_{n}(A)\rvert$. For such purposes, some notation is required. We define for each fixed $\bfr\subset \mathbb{N}\cap [1,n]$ of cardinality $1\leq \lvert \bfr\rvert\leq 2g+4$ the sets \begin{equation*}\label{Tnr} T^{0}_{n,\bfr}=\Big\{\bfx\in T_{n}:\ \ \ \ \ \ \ \ \ \ \{x_{1},x_{2},x_{3}\}\cap\bfr=\o,\ \ \ \ \ \ \ \ \ \bfr\subset \text{Set}(\bfx)\ \Big\},\end{equation*}
\begin{equation*}\label{Tnr1}T^{1}_{n,\bfr}= \Big\{\bfx\in T_{n}:\ \ \ \ \ \ \ \ 1\leq \lvert\{x_{1},x_{2},x_{3}\}\cap\bfr\lvert\leq 2,\ \ \ \ \ \ \bfr\subset \text{Set}(\bfx)\ \Big\},\end{equation*}
\begin{equation}\label{Tnre3}T^{3}_{n,\bfr}= \Big\{\bfx\in T_{n}:\ \ \ \ \ x_{j}\in\bfr\ \ \ \ \text{   for every $1\leq j\leq 3$},\ \ \bfr\subset \text{Set}(\bfx)\ \Big\},\end{equation} and when $i\in\{0,1,3\}$ for given $A\subset \mathbb{N}$ the corresponding counterparts $$T^{i}_{n,\bfr}(A)=\Big\{\bfx\in T^{i}_{n,\bfr}:\ \ \ \ \ \text{Set}(\bfx)\subset A\ \Big\}.$$  
\begin{lem}\label{lem7.1}
For fixed $\mathcal{R}\subset \mathbb{N}\cap [1,n]$ with $1\leq \lvert\bfr\rvert\leq 2g+3$ one has that
\begin{equation}\label{alasa}\mathbb{E}\big(\lvert T^{i}_{n,\bfr}(A)\rvert \ \big |\bfr\subset A\big)\ll n^{-\delta_{g}},\ \ \ \ \ \ \ \ i=0,1\end{equation} wherein $\delta_{g}>0$ is a fixed constant which does not depend on $\bfr$.
\end{lem}
\begin{proof}
We shall begin by analysing $\mathbb{E}\big(\lvert T^{0}_{n,\bfr}(A)\rvert \ \big|\bfr\subset A\big)$, and observe that then
\begin{align}\label{pelele}n^{2\alpha+\varepsilon\alpha}\mathbb{E}&\big(\lvert T^{0}_{n,\bfr}(A)\rvert \ \big|\bfr\subset A\big)\ll \mathop{{\sum_{\substack{\{x_{1},x_{2},x_{3}\}\in R_{n}\\ x_{1}+x_{4}=x_{5}+r_{6}\\ x_{1}>r_{6}}}}^*}(x_{4}x_{5})^{-\alpha}+\mathop{{\sum_{\substack{\{x_{1},x_{2},x_{3}\}\in R_{n}\\ x_{1}+r_{4}=x_{5}+x_{6}}}}^*}(x_{5}x_{6})^{-\alpha}\nonumber
\\
&+\mathop{{\sum_{\substack{\{x_{1},x_{2},x_{3}\}\in R_{n}\\ x_{1}+r_{4}=x_{5}+r_{6}\\ \{r_{4},r_{6}\}\neq \{x_{2},x_{3}\}}}}^*}x_{5}^{-\alpha}+\mathop{{\sum_{\substack{\{x_{1},x_{2},x_{3}\}\in R_{n}\\ x_{1}+x_{4}=r_{5}+r_{6}\\ r_{5}\in \bfr}}}^*}x_{4}^{-\alpha}+\mathop{{\sum_{\substack{\{x_{1},x_{2},x_{3}\}\in R_{n}\\ x_{1}+r_{4}=r_{5}+r_{6}\\ \{x_{2},x_{3}\}\not\subset \{r_{4},r_{5},r_{6}\}}}}^*}1,
\end{align}
where in the above lines $r_{4}\in\bfr\cup \{x_{1},x_{2},x_{3}\}$ and $r_{j}\in\bfr\cup \{x_{2},x_{3}\}$ for $5\leq j\leq 6$ and where we employed the fact that the implicit inner sums corresponding to the rest of the underlying pairs of the shape $(x_{2j+1},x_{2j+2})$ for $3\leq j\leq g+1$ in (\ref{jood}) are $O(1)$ by Lemma \ref{lem1.1} regardless of whether $\bfr \cap \{x_{2j+1},x_{2j+2}\}$ is non-trivial or empty. The restriction in the last three summands stems from the discussion in (\ref{gui}).  We may as well assume that $x_{5}\neq x_{6}$ since the equality $x_{2j+1}=x_{2j+2}$ can only hold for at most one $2\leq j\leq g+1$. We then use Lemma \ref{lem1.1} to deduce that
\begin{align*}\sum_{\substack{\{x_{1},x_{2},x_{3}\}\in R_{n}\\ x_{1}+x_{4}=x_{5}+r_{6}\\ x_{1}>r_{6}\\ r_{6}\in\bfr\cup \{x_{2},x_{3}\}}}(x_{4}x_{5})^{-\alpha}\ll& \sum_{\substack{\{x_{1},x_{2},x_{3}\}\in R_{n}\\ x_{1}>r_{6}\\r_{6}\in\bfr\cup \{x_{2},x_{3}\}}}(x_{1}-r_{6})^{1-2\alpha}\ll \sum_{\substack{x_{2}\asymp n\\ r_{6}<x_{1}\leq g_{\lambda(n)}\\ r_{6}\in\bfr}}(x_{1}-r_{6})^{1-2\alpha}\nonumber
\\
&+\sum_{\substack{r_{6}<x_{1}\asymp n\\ x_{2}\leq g_{\lambda(n)}\\ r_{6}\in\bfr\cup \{x_{2},x_{3}\}}}(x_{1}-r_{6})^{1-2\alpha}\ll (\log n)^{B}(n^{(2-2\alpha)\varepsilon+1}+n^{2-2\alpha+\varepsilon}).
\end{align*}
We shall as well utilise Lemma \ref{lem1.1} in conjunction with (\ref{ela}) to obtain
\begin{align*}n^{-2\alpha-\varepsilon\alpha}\mathop{{\sum_{\substack{\{x_{1},x_{2},x_{3}\}\in R_{n}\\ x_{1}+r_{4}=x_{5}+x_{6}}}}^*}(x_{5}x_{6})^{-\alpha}\ll& n^{-2\alpha-\varepsilon\alpha}\sum_{\substack{\{x_{1},x_{2},x_{3}\}\in R_{n}\\ r_{4}\in \bfr\cup\{x_{1},x_{2},x_{3}\}}}(x_{1}+r_{4})^{1-2\alpha}\ll_{\lambda} n^{(1-2\alpha)\varepsilon}(\log n).
\end{align*}

The third summand in (\ref{pelele}) may be estimated in an similar way as the first one, namely
\begin{align*}\mathop{{\sum_{\substack{\{x_{1},x_{2},x_{3}\}\in R_{n}\\ x_{1}+r_{4}=x_{5}+r_{6}\\ \{r_{4},r_{6}\}\neq \{x_{2},x_{3}\}}}}^*}x_{5}^{-\alpha}&\ll \mathop{{\sum_{\substack{\{x_{1},x_{2},x_{3}\}\in R_{n}\\ \ x_{1}>r_{6}-r_{4}\\ \{r_{4},r_{6}\}\neq \{x_{2},x_{3}\}}}}^*}(x_{1}+r_{4}-r_{6})^{-\alpha}\ll_{\lambda} (\log n)^{B}(n^{(1-\alpha)\varepsilon+1}+n^{1-\alpha+\varepsilon}),
\end{align*} where in the above line $r_{4}\in \bfr\cup\{x_{1},x_{2},x_{3}\}$ and $r_{6}\in \bfr\cup\{x_{2},x_{3}\}$ and if $r_{4}=x_{1}$ then $r_{6}\in\bfr$ by the argument in (\ref{gui}). It seems worth clarifying that the above linear expression in brackets can always be written after possibly using the identity $x_{1}+x_{2}+x_{3}=n$ as a linear expression involving only one variable in $\{x_{1},x_{2},x_{3}\}$ or two of them, say $x_{1},x_{2},$ with $x_{2}\leq g_{\lambda}(n)$, the third one satisfying $x_{3}\asymp n$ and being determined once the others are fixed. For the fourth summand,
\begin{align*}\mathop{{\sum_{\substack{\{x_{1},x_{2},x_{3}\}\in R_{n}\\ x_{1}+x_{4}=r_{5}+r_{6}\\ r_{5}\in \bfr}}}^*}x_{4}^{-\alpha}&\ll \mathop{{\sum_{\substack{\{x_{1},x_{2},x_{3}\}\in R_{n}\\ x_{1}<r_{5}+r_{6}\\ r_{5}\in \bfr}}}^*}(r_{5}+r_{6}-x_{1})^{-\alpha}\ll_{\lambda} (\log n)^{B}(n^{(1-\alpha)\varepsilon+1}+n^{1-\alpha+\varepsilon}),
\end{align*} where $r_{6}\in \bfr\cup\{x_{2},x_{3}\}$ and we employed a similar argument as above. Lastly,
$$n^{-2\alpha-\varepsilon\alpha}\mathop{{\sum_{\substack{\{x_{1},x_{2},x_{3}\}\in R_{n}\\ x_{1}+r_{4}=r_{5}+r_{6}\\ \{x_{2},x_{3}\}\not\subset \{r_{4},r_{5},r_{6}\}}}}^*}1\ll n^{-2\alpha-\varepsilon\alpha}\mathop{{\sum_{\substack{x_{1}+x_{3}=n-x_{2}\\ x_{1}+r_{4}=r_{5}+r_{6}\\  \{x_{2},x_{3}\}\not\subset \{r_{4},r_{5},r_{6}\}}}}^*}1\ll n^{1-2\alpha-\varepsilon\alpha},$$ wherein we used the fact under the above constraints that in the recalcitrant instance $r_{4}=x_{3}$ or $ r_{4}=x_{2}$ then $r_{5},r_{6}\in\bfr,$ and hence one would be able to express $x_{2}$ or $x_{3}$ respectively only in terms of $\bfr$. By the preceding discussion 
$$\mathbb{E}\big(\lvert T^{0}_{n,\bfr}(A)\rvert\ \big|\bfr\subset A\big)\ll n^{-\delta_{g}}.$$ 

We next focus our attention on the case $i=1$, recall (\ref{feo}) and observe that
\begin{align}\label{kkp}\sum_{\substack{\{x_{1},x_{2},x_{3}\}\in R_{n}\\ 1\leq\lvert\{x_{1},x_{2},x_{3}\}\cap\bfr\lvert\leq 2}}\mathbb{P}&\big(x_{1},x_{2},x_{3}\in A| \bfr\subset A\big)\ll \sum_{\substack{x_{2}+x_{3}=n-r_{1}\\ r_{1}\asymp g_{\lambda}(n)\\ r_{1}\in\bfr}}(x_{2}x_{3})^{-\alpha}+\sum_{\substack{\{r_{1},x_{2},x_{3}\}\in R_{n}\\ r_{1}\asymp n\\ r_{1}\in\bfr}}(x_{2}x_{3})^{-\alpha}\nonumber
\\
&+\sum_{\substack{\{r_{1},r_{2},x_{3}\}\in R_{n}\\ r_{1},r_{2}\in\bfr}}x_{3}^{-\alpha}\ll_{\lambda} n^{1-2\alpha}+n^{-\alpha+\varepsilon(1-\alpha)}\log n+n^{-\alpha\varepsilon}\ll n^{-\delta_{g}}.\end{align}
We then split the corresponding computation into several parts, namely
\begin{align*}\mathbb{E}\big(\lvert T^{1}_{n,\bfr}(A)\rvert\ \big |\bfr\subset A\big)\ll &\sum_{\substack{\{x_{1},x_{2},x_{3}\}\in R_{n}\\ 1\leq\lvert\{x_{1},x_{2},x_{3}\}\cap\bfr\lvert\leq 2}}\mathbb{P}\big(x_{1},x_{2},x_{3}\in A| \bfr\subset A\big)\Bigg(\sum_{\substack{x_{1}+x_{4}=x_{5}+r_{6}\\ r_{6}\in\bfr\cup \{x_{2},x_{3}\}}}(x_{4}x_{5})^{-\alpha}\nonumber
\\
&+\sum_{\substack{ x_{1}+x_{4}=\ldots=x_{2g+3}+x_{2g+4}\\ x_{4}\leq x_{1}}}\prod_{j=4}^{2g+4}x_{j}^{-\alpha}+\sum_{\substack{x_{1}+r_{4}=x_{5}+x_{6}\\ r_{4}\in\bfr\cup \{x_{1},x_{2},x_{3}\}}}(x_{5}x_{6})^{-\alpha}\nonumber
\\
&+\sum_{\substack{x_{1}+x_{4}=r_{5}+r_{6}\\ r_{5},r_{6}\in\bfr\cup \{x_{2},x_{3}\}}}x_{4}^{-\alpha}+\sum_{\substack{x_{1}+r_{4}=r_{5}+x_{6}\\ r_{4},r_{5}\in\bfr\cup \{x_{1},x_{2},x_{3}\}}}x_{6}^{-\alpha}+\sum_{\substack{x_{1}+r_{4}=r_{5}+r_{6}\\ r_{j}\in\bfr\cup \{x_{1},x_{2},x_{3}\}\\ 4\leq j\leq 6}}1\Bigg).
\end{align*}
One may avoid encompassing the instance in which $x_{5}=x_{6}$ in the same manner as was detailed after (\ref{pelele}). We employ first Lemma \ref{lem1.1} and (\ref{esti}) to deduce that the second summand is bounded above by a constant times
\begin{align*}\sum_{\substack{\{x_{1},x_{2},x_{3}\}\in R_{n}\\ 1\leq\lvert\{x_{1},x_{2},x_{3}\}\cap\bfr\lvert\leq 2}}\mathbb{P}&\big(x_{1},x_{2},x_{3}\in A| \bfr\subset A\big)\sum_{\substack{ x_{4}\leq x_{1}}}x_{4}^{-\alpha}(x_{1}+x_{4})^{(1-2\alpha)g}
\\
&\ll \sum_{\substack{\{x_{1},x_{2},x_{3}\}\in R_{n}\\ 1\leq\lvert\{x_{1},x_{2},x_{3}\}\cap\bfr\lvert\leq 2}}\mathbb{P}\big(x_{1},x_{2},x_{3}\in A| \bfr\subset A\big)\ll_{\lambda}n^{-\delta_{g}},
\end{align*}
where in the last step we employed (\ref{kkp}). The rest of the summands are bounded either trivially or via Lemma \ref{lem1.1} by a constant times $$\sum_{\substack{\{x_{1},x_{2},x_{3}\}\in R_{n}\\ 1\leq\lvert\{x_{1},x_{2},x_{3}\}\cap\bfr\lvert\leq 2}}\mathbb{P}\big(x_{1},x_{2},x_{3}\in A| \bfr\subset A\big),$$ whence (\ref{kkp}) combined with the preceding estimates yields (\ref{alasa}) for the instance $i=1$. 
\end{proof}

The conditional expectation of the random variable defined in (\ref{Tnre3}) when $i=3$ on the contrary does not exhibit a power saving, the ensuing extensive analysis articulated hereafter having its genesis on such a fact. It also seems worth introducing for $n\in\mathbb{N}$ and $\bfr\subset \mathbb{N}\cap [1,n]$ the set
\begin{equation*}T^{4'}_{n,\bfr}= \bigcup_{k_{0}=4}^{2g+4}\Big\{\bfx\in T^{3}_{n,\bfr}:\ x_{k_{0}}\in\bfr,\ \ \ \ \ x_{2j+1},x_{2j+2}\notin \bfr,\ \ \text{for some $2\leq j\leq g+1$} \Big\},\end{equation*} defining
\begin{equation*}T^{4''}_{n,\bfr}= \Big\{\bfx\in T^{3}_{n,\bfr}:\  x_{2j+1}=x_{2j+2}\notin \bfr,\ \ \text{for some $2\leq j\leq g+1$} \Big\},\ \ \  \end{equation*} and $$T^{4}_{n,\bfr}=T^{4'}_{n,\bfr}\cup T^{4''}_{n,\bfr}\ \ \ \ \ \ \ \ \ \ \ \ \ \ T^{5}_{n,\bfr}=T^{3}_{n,\bfr}\setminus T^{4}_{n,\bfr},$$
 and for $A\subset\mathbb{N}$ considering the random variables \begin{equation}\label{Tnr3n}T^{i}_{n,\bfr}(A)=\Big\{\bfx\in  T^{4}_{n,\bfr}:\ \ \text{Set}(\bfx)\subset A\Big\}\ \ \ \ \ \ \ i=4,5.\end{equation}It may be worth observing that if $\bfx\in T^{5}_{n,\bfr}$ then either $\lvert\bfr\rvert=3$ or $g+2\leq \lvert\bfr\rvert\leq 2g+4,$ the latter situation entailing that either $x_{2l}\in\bfr$ or $x_{2l-1}\in\bfr$ for each $3\leq l\leq g+2$. The case $\lvert \bfr\rvert=g+2$ corresponds to the instance $\bfx\in T_{n}^{dep}$, it being profitable recalling the argument in (\ref{gui}). If $\lvert \bfr\rvert=3$ and $\bfx\in T^{5}_{n,\bfr}$ then the same argument implies that $\bfr=\{x_{1},x_{2},x_{3}\}$ and $\bfr\cap\text{Set}(x_{4},\ldots,x_{2g+4})=\o.$
\begin{lem}\label{lem7.2}
For fixed $\bfr\subset \mathbb{N}\cap [1,n]$ satisfying $3\leq \lvert \bfr\rvert\leq 2g+4$ one has that
\begin{equation*}\label{alas}\mathbb{E}\big(\lvert T^{4}_{n,\bfr}(A)\rvert \ \big|\bfr\subset A\big)\ll n^{-\delta_{g}},\ \ \ \ \ \ \ \ \mathbb{E}\big(\lvert T^{5}_{n,\bfr}(A)\rvert \ \big|\bfr\subset A\big)\ll 1 \end{equation*}for some $\delta_{g}>0$, wherein the implicit constant does not depend on $\bfr$.
\end{lem}
\begin{proof}
We shall distinguish in the analysis concerning $T^{5}_{n,\bfr}(A)$ as above between the instance when $x_{i}\in\bfr$ for some $i\geq 4$ and when that property does not hold, namely
\begin{align}\label{pelelee25}\mathbb{E}\big(\lvert T^{5}_{n,\bfr}(A)\rvert\ \big|\bfr\subset A\big)\ll &1+\sum_{\substack{\{r_{1},r_{2},r_{3}\}\in R_{n}\\ r_{1}+x_{4}=\ldots=x_{2g+3}+x_{2g+4}\\ x_{4}\leq r_{1}}}\prod_{i=4}^{2g+4}x_{i}^{-\alpha}+\sum_{\substack{\{r_{1},r_{2},r_{3}\}\in R_{n}\\ r_{1}+x_{4}=x_{5}+r_{6}\\ r_{6}\in\bfr}}(x_{4}x_{5})^{-\alpha},
\end{align}
where $r_{j}\in\bfr$ for $1\leq j\leq 3$ and the first summand encompasses the instance in which all the variables are uniquely determined by $\bfr$. We estimate the second one via Lemma \ref{lem1.1} to obtain
$$\sum_{\substack{\{r_{1},r_{2},r_{3}\}\in R_{n}\\ r_{1}+x_{4}=\ldots=x_{2g+3}+x_{2g+4}\\ x_{4}\leq r_{1}}}\prod_{j=4}^{2g+4}x_{j}^{-\alpha}\ll \sum_{x_{4}\leq r_{1}}(r_{1}+x_{4})^{(1-2\alpha)g}x_{4}^{-\alpha}\ll r_{1}^{(1-2\alpha)g}\sum_{x_{4}\leq r_{1}}x_{4}^{-\alpha}\ll 1.$$ The third one in (\ref{pelelee25}) is bounded via a routine application of Lemma \ref{lem1.1}. In contrast,
\begin{align*} \mathbb{E}\big(\lvert T^{4}_{n,\bfr}(A)\rvert \ \big|\bfr\subset A\big)\ll&  \sum_{\substack{\{r_{1},r_{2},r_{3}\}\in R_{n}\\ r_{1}+x_{4}=r_{5}+x_{6}=x_{2g+3}+x_{2g+4}\\ x_{4}\leq r_{1}}}(x_{4}x_{6}x_{2g+3}x_{2g+4})^{-\alpha}+\sum_{\substack{\{r_{1},r_{2},r_{3}\}\in R_{n}\\ r_{1}+r_{4}=x_{5}+x_{6}}}(x_{5}x_{6})^{-\alpha}
\\
&+r_{1}^{-\alpha}+ \sum_{\substack{\{r_{1},r_{2},r_{3}\}\in R_{n}\\ x_{4}\leq r_{1}}}x_{4}^{-\alpha}\sum_{r_{1}+x_{4}=2x_{2g+4}}x_{2g+4}^{-\alpha},
\end{align*}
wherein $r_{j}\in\bfr$ for $1\leq j\leq 5$. A routinary application of Lemma \ref{lem1.1} then delivers
\begin{align*} \mathbb{E}\big(\lvert T^{4}_{n,\bfr}(A)\rvert \ \big|\bfr\subset A\big)\ll& r_{1}^{(1-2\alpha)}\big(1+ (r_{1}-r_{5})^{1-2\alpha}\big)\ll n^{(1-2\alpha)\varepsilon}.                              
\end{align*}

\end{proof}

\section{Random variables comprising double sums}\label{sec11}
As foreshadowed above, estimates for the upper tail of $\lvert T_{n}(A)\rvert$ in the present memoir have their reliance on the computation of $\mathbb{E}(\lvert T_{n}(A)\rvert^{m})$ for some sufficiently large $m\in\mathbb{N}$, the conclusion stemming from Lemma \ref{lem7} lending credibility to the expectation that such a moment may be of order $(\log n)^{m}$. When expanding the preceding moment appropiately, expressing it as the sum over $m$ tuples and sorting the sums according to the intersection of the set of elements of the tuples, it will transpire that for a fixed position such an intersection may be as large as $O(m^{2g+4}),$ whence a naive approach employing the preceding bound would lead to an estimate no better than $$\mathbb{E}(\lvert T_{n}(A)\rvert^{m})\ll m^{(2g+4)m},$$ which when taking $m\sim \log n$ as we will leads to an undesirably large estimate. The analysis in the preceding section enabled one to deduce that the contribution of most of the individual sums is sufficiently small, the end of the present one being to derive an analogous conclusion for most of the pairs of sums running over tuples having non-empty intersections. It seems worth defining beforehand for fixed $r_{1}\gg n^{\varepsilon}$ the set
\begin{equation}\label{Vn}V_{r_{1}}=\Big\{(x_{4},x_{5},\ldots,x_{2g+4})\in\mathbb{N}^{2g+1}:\ \ r_{1}+x_{4}=\ldots=x_{2g+3}+x_{2g+4},\ \ x_{i}\leq r_{1}\Big\},\end{equation}
where in the above line the components of $(r_{1},x_{4},x_{5},\ldots,x_{2g+4})$ are pairwise disjoint, and introduce for further convenience and $\bfr\subset\mathbb{N}\cap [1,n]$ with $0\leq \lvert \bfr \rvert\leq 2g+3$ the random variable
\begin{equation}\label{tama}X_{n,\bfr}(A)=\sum_{\substack{\bfx\in T_{n}^{ind}\\ \text{Set}(\bfx)\cap \bfr=\o}}\leavevmode\hbox{$1\!\rm I$}_{A}(\bfx)\Bigg(\sum_{\substack{\o\neq \overline{\bfx}\subset\bfx\\ }}\lvert T_{n,\bfr\cup \overline{\bfx}}^{5}(A)\rvert-\leavevmode\hbox{$1\!\rm I$}_{\o}(\bfr)\Big(1+\sum_{l=1}^{3}\sum_{\substack{\bfy\in V_{x_{l}}\\ \text{Set}(\bfx)\cap\text{Set}(\bfy)=\o}}\leavevmode\hbox{$1\!\rm I$}_{A}(\bfy)\Big)\Bigg),\end{equation}
wherein $\leavevmode\hbox{$1\!\rm I$}_{\o}(\bfr)=1$ if $\bfr=\o$ and $0$ otherwise, the function $\leavevmode\hbox{$1\!\rm I$}_{A}(\bfx)=1$ whenever $\text{Set}(\bfx)\subset A$ and $0$ else, and $x_{l}$ denotes the $l$-th component of the vector $\bfx$. We deem it convenient computing its conditional expectation, it being profitable presenting first a technical lemma that shall be employed throughout this section. We also introduce for $1\leq k\leq 3$, given $\bfx \in T_{n}^{ind}$ and fixed vector $\boldsymbol{\varepsilon}=(\varepsilon_{1},\varepsilon_{2},\varepsilon_{3})\in\{-1,1\}^{3}$ the set
$$\mathcal{T}_{\bfx,\bfr,\boldsymbol{\varepsilon}}=\Big\{\bfy\in\mathbb{N}^{k}:\ \  y_{j}\in\text{Set}(\bfx)\cup\bfr,\ \ 1\leq j\leq k,\ \ \text{Set}(\bfx)\cap\text{Set}(\bfy)\neq \o,\ \ \ y_{i}=y_{j}\Rightarrow \varepsilon_{i}=\varepsilon_{j}\Big\}.$$
\begin{lem}\label{lem32}
Let $n\in\mathbb{N}$ and $\bfr\subset \mathbb{N}\cap [1,n]$ be a fixed subset. Take a constant $\beta>0$, let $1\leq k\leq 3$, take $\bfy=(y_{1},\ldots,y_{k})\in\mathbb{N}^{k}$ and consider the function $F(\bfy)=\sum_{j=1}^{k}\varepsilon_{j}y_{j}.$ Then there is some constant $\delta_{\beta}>0$ for which
$$\sum_{\substack{\bfx\in T_{n}^{ind}\\ \text{Set}(\bfx)\cap \bfr=\o}}\prod_{i=1}^{2g+4}x_{i}^{-\alpha}\sum_{\substack{\bfy\in\mathcal{T}_{\bfx,\bfr,\boldsymbol{\varepsilon}}\\ F(\bfy)\geq 1}}F(\bfy)^{-\beta}\ll n^{-\delta_{\beta}}.$$ Likewise, one has for fixed $r_{1}\gg n^{\varepsilon}$ the bound
\begin{equation}\label{lasi}\sum_{\substack{\bfx\in V_{r_{1}}\\ \text{Set}(\bfx)\cap \bfr=\o}}\prod_{i=4}^{2g+4}x_{i}^{-\alpha}\sum_{\substack{ \bfy\in\mathcal{T}_{\bfx,\bfr,\boldsymbol{\varepsilon}}\\ F(\bfy)\geq 1}}F(\bfy)^{-\beta}\ll n^{-\delta_{\beta}}.\end{equation} 
\end{lem}
\begin{proof}
We focus our attention on proving the first estimate, begin by assuming that \begin{equation}\label{Tim}F(\bfy)\asymp (x_{4}-\tilde{z})\end{equation} for some $\tilde{z}=\tilde{z}(x_{1},\bfr)$ that only depends on $x_{1}$ and $\mathcal{R}$ and observe that
\begin{align*}\sum_{\substack{\bfx\in T_{n}^{ind}\\ F(\bfy)\asymp (x_{4}-\tilde{z})}}\prod_{i=1}^{2g+4}x_{i}^{-\alpha}F(\bfy)^{-\beta}\ll& n^{-\varepsilon\beta/10}\sum_{\substack{\bfx\in T_{n}^{ind}\\ x_{4}-\tilde{z}>n^{\varepsilon/10}}}\prod_{i=1}^{2g+4}x_{i}^{-\alpha}+\sum_{\substack{\bfx\in T_{n}^{ind}\\ x_{4}-\tilde{z}\leq n^{\varepsilon/10}}}\prod_{i=1}^{2g+4}x_{i}^{-\alpha}.
\end{align*}
We note first by an application of Lemma \ref{lem7} that
\begin{equation}\label{aspe}n^{-\varepsilon\beta/10}\sum_{\substack{\bfx\in T_{n}^{ind}\\ x_{4}-\tilde{z}>n^{\varepsilon/10}}}\prod_{i=1}^{2g+4}x_{i}^{-\alpha}\ll n^{-\varepsilon\beta/10}\mathbb{E}(\lvert T_{n}^{ind}(A)\rvert)\ll_{\lambda}n^{-\varepsilon\beta/10}\log n. \end{equation}

A routinary application of Lemma \ref{lem1.1} then delivers the bound
 \begin{align}\label{men}\sum_{\substack{\bfx\in T_{n}^{ind}\\ 0<x_{4}-\tilde{z}\leq n^{\varepsilon/10}}}&\prod_{i=1}^{2g+4}x_{i}^{-\alpha}\ll \sum_{\substack{\{x_{1},x_{2},x_{3}\}\in R_{n}}}x_{1}^{(1-2\alpha)g-\alpha}(x_{2}x_{3})^{-\alpha}\sum_{\substack{x_{4}\leq x_{1}\\ 0<x_{4}-\tilde{z}\leq n^{\varepsilon/10}}}x_{4}^{-\alpha}\nonumber
\\
&\ll n^{\varepsilon(1/10+(1-2\alpha)g)}\sum_{\substack{\{x_{1},x_{2},x_{3}\}\in R_{n}}}(x_{1}x_{2}x_{3})^{-\alpha}\ll_{\lambda} n^{\varepsilon(1/10+(1-2\alpha)g)}\log n,
\end{align}
wherein we employed (\ref{ela}). If instead $F(\bfy)\asymp x_{1}+g(\bfr)$ for some $g(\bfr)$ only depending on $\bfr$ then a customary application of Lemma \ref{lem1.1} and equations (\ref{ela}) and (\ref{esti}) would entail
 \begin{align}\label{pasti}\sum_{\substack{\bfx\in T_{n}^{ind}\\ 0<x_{1}+g(\bfr)\leq n^{\varepsilon/10}}}&\prod_{i=1}^{2g+4}x_{i}^{-\alpha}\ll \sum_{\substack{\{x_{1},x_{2},x_{3}\}\in R_{n}\\ 0<x_{1}+g(\bfr)\leq n^{\varepsilon/10}}}(x_{1}x_{2}x_{3})^{-\alpha}\ll_{\lambda} (\log n)n^{\varepsilon(1/10-\alpha)+1-2\alpha},
\end{align} which in combination with an analogous estimate to the one in (\ref{aspe}) would deliver the desired bound. It therefore transpires that if $y_{j}\in\{x_{1},x_{4}\}\cup\bfr$ for $1\leq j\leq k\leq 3$ then one could either express $F(\bfy)$ as in (\ref{Tim}) or right before (\ref{pasti}), whence the preceding discussion would deliver the required estimate.

For the remaining cases we may assume that $F(\bfy)\neq x_{1}+x_{2}+x_{3}=n,$ the latter case being derived after an application of Lemma \ref{lem7}, and observe that by making the substitution $x_{3}=n-x_{1}-x_{2}$ and denoting $\tilde{\bfx}=(x_{4},\ldots,x_{2g+4})$ one may write $$F(\bfy)=\delta_{1}x_{1}+\delta_{2}x_{2}+g(\bfr,\tilde{\bfx})$$ for some integers $-2\leq\delta_{1},\delta_{2}\leq 2$ and a function $g(\bfr,\tilde{\bfx})$. When $\delta_{2}\neq 0$ then $F(\bfy)\asymp (x_{2}-\tilde{x})$, wherein $\tilde{x}<x_{2}$ does not involve $x_{2}$ or $x_{3}$ on its expression. Such an observation then entails
\begin{align*}\label{apa}\sum_{\substack{\bfx\in T_{n}^{ind}\\ F(\bfy)\asymp (x_{2}-\tilde{x})}}\prod_{i=1}^{2g+4}x_{i}^{-\alpha}F(\bfy)^{-\beta}\ll &n^{-\varepsilon\beta/10}\sum_{\substack{\bfx\in T_{n}^{ind}\\ x_{2}-\tilde{x}>n^{\varepsilon/10}}}\prod_{i=1}^{2g+4}x_{i}^{-\alpha}+ \sum_{\substack{\bfx\in T_{n}^{ind}\\ 0<x_{2}-\tilde{x}\leq n^{\varepsilon/10}}}\prod_{i=1}^{2g+4}x_{i}^{-\alpha}.
\end{align*}
We estimate the first term as in (\ref{aspe}), and for the second one
\begin{align*}\sum_{\substack{\bfx\in T_{n}^{ind}\\ 0<x_{2}-\tilde{x}\leq n^{\varepsilon/10}}}\prod_{i=1}^{2g+4}x_{i}^{-\alpha}&\ll \sum_{\substack{x_{1}+x_{4}=\ldots=x_{2g+3}+x_{2g+4}\\ x_{4}<x_{1}<n}}\prod_{i=4}^{2g+4}x_{i}^{-\alpha}\sum_{\substack{ x_{3}=n-x_{1}-x_{2}\\ 0<x_{2}-\tilde{x}\leq n^{\varepsilon/10}}}(x_{1}x_{2}x_{3})^{-\alpha}
\\
&\ll n^{\varepsilon(1/10-\alpha)-2\alpha}\sum_{\substack{x_{1}+x_{4}=\ldots=x_{2g+3}+x_{2g+4}\\ x_{4}<x_{1}<n}}\prod_{i=4}^{2g+4}x_{i}^{-\alpha}\ll_{\lambda} n^{\varepsilon(1/10-\alpha)+1-2\alpha},
\end{align*}
where in the last step we employed Lemma \ref{lem1.1} and crutially utilised the fact that the interval for $x_{2}$ does not depend on $x_{3}$. 

If $\delta_{2}=0$ and $x_{2l+1}=y_{j_{0}}$ for some $1\leq j_{0}\leq k$ and $2\leq l\leq g+1$ (and similarly with even indexes) with the property that $y_{j}\neq x_{2l+2}$ for all $1\leq j\leq k$ then one may write $F(\bfy)\asymp(x_{2l+1}-\tilde{y})$, wherein $\tilde{y}$ does not involve $x_{2l+1}$ or $x_{2l+2}$ on its expression and obtain
\begin{align*}\sum_{\substack{\bfx\in T_{n}^{ind}\\ F(\bfy)\asymp(x_{2l+1}-\tilde{y})}}(x_{1}\cdots x_{2g+4})^{-\alpha}F(\bfy)^{-\beta}\ll_{\lambda} n^{-\varepsilon\beta/10}\log n+\sum_{\substack{\bfx\in T_{n}^{ind}\\ 0<x_{2l+1}-\tilde{y}\leq n^{\varepsilon/10}}}\prod_{j=1}^{2g+4}x_{j}^{-\alpha}.
\end{align*}
We assume for simplicity and without loss of generality that $l=g+1$ and get the estimate \begin{align*}\sum_{\substack{\bfx\in T_{n}^{ind}\\ 0<x_{2g+3}-\tilde{y}\leq n^{\varepsilon/10}}}&\prod_{i=1}^{2g+4}x_{i}^{-\alpha}\ll \sum_{\substack{\{x_{1},x_{2},x_{3}\}\in R_{n}\\ x_{1}+x_{4}=\ldots=x_{2g+1}+x_{2g+2}\\ x_{4}\leq x_{1}}}\prod_{i=1}^{2g+2}x_{i}^{-\alpha}\sum_{\substack{x_{2g+4}=x_{1}+x_{4}-x_{2g+3}\\ 0<x_{2g+3}-\tilde{y}\leq n^{\varepsilon/10}}}(x_{2g+3}x_{2g+4})^{-\alpha}
\\
&\ll n^{\varepsilon/10}\sum_{\substack{\{x_{1},x_{2},x_{3}\}\in R_{n}\\ x_{1}+x_{4}=\ldots=x_{2g+1}+x_{2g+2}\\ x_{4}\leq x_{1}}}x_{1}^{-2\alpha}\prod_{i=2}^{2g+2}x_{i}^{-\alpha}
\\
&\ll n^{\varepsilon/10}\sum_{\substack{\{x_{1},x_{2},x_{3}\}\in R_{n}}}x_{1}^{(1-2\alpha)g}(x_{1}x_{2}x_{3})^{-\alpha}\ll _{\lambda}n^{\varepsilon(1/10+(1-2\alpha)g)}\log n,
\end{align*}
wherein we employed the fact that $x_{2g+3}x_{2g+4}\gg x_{1}$ in the second step and the fact that $\tilde{y}$ does not depend on $x_{2g+4}$, Lemma \ref{lem1.1} in the third one and (\ref{ela}) in the fourth one, as desired. We observe in view of (\ref{alpha}) that $(1-2\alpha)g=-g/(2g+1)$ and hence the above exponent is negative. We also note that an analogous procedure to the preceding one would deliver the estimate (\ref{lasi}).

If $\delta_{2}=0$ and instead $g(\bfr,\tilde{\bfx})$ is not in the previously described situation and contains both $x_{2g+3}$ and $x_{2g+4}$ as summands then either they have a different coefficient, in which case the substitution $x_{2g+4}=x_{1}+x_{4}-x_{2g+3}$ would reduce it to the previously discussed situation, or they have the same one. In the latter instance then 
$$F(\bfy)=\delta_{1}x_{1}+\xi(x_{2g+3}+x_{2g+4})+\delta_{4}x_{4}+\delta_{\bfr}r,\ \ \ \ \ \ \ \ \ \ r\in\bfr,$$ where $\xi\in\{-1,1\}$ and $\delta_{1},\delta_{4},\delta_{\bfr}\in\{-1,0,1\}$ with the property that at least two of these coefficients are $0$. By making the aforementioned substitution it would transpire that \begin{equation*}F(\bfy)=\gamma_{1}x_{1}+\gamma_{4}x_{4}+\gamma_{\bfr}r\end{equation*} for some integers $-2\leq\gamma_{1},\gamma_{4},\gamma_{\bfr}\leq 2$ satisfying that either $\gamma_{4}\neq 0$ or $\gamma_{4}=0$ but $\gamma_{1}\neq 0$ and $\gamma_{\bfr}=0$. The first case would correspond to the instance for which (\ref{Tim}) holds, whence (\ref{men}) would deliver the desired estimate. The second one would follow via (\ref{pasti}).  The estimate (\ref{lasi}) would be derived in an analogous manner by making the additional observation in view of the above analysis that either the relation (\ref{Tim}) holds or $F(\bfy)\asymp r_{1}\gg n^{\varepsilon}.$ The preceding discussion completes the proof.
\end{proof}

We next define for further convenience and $3\leq h\leq 4$, fixed $\mathcal{R}\subset\mathbb{N}$ and vectors of the shape $\bfx=(x_{1},\ldots,x_{k})\in\mathbb{N}^{k}$ for $1\leq k\leq 2g+4$ the sets
\begin{equation}\label{rac1}\mathcal{C}^{h}_{\bfr,\bfx}=\Big\{\{y_{1},\ldots,y_{h}\}\subset\text{Set}(\bfx)\cup \bfr,\ \ \ \ y_{h}\in \text{Set}(\bfx)\setminus \bfr \ \ \Big\}\end{equation} and
$$\mathcal{C}_{\bfr,\bfx}=\Big\{\{y_{1},y_{2},y_{3},y_{4}\}\in \mathcal{C}^{4}_{\bfr,\bfx}:\ \ \ \ \ y_{1}+y_{2}=y_{3}+y_{4},\ \  \ \ \{y_{1},y_{2}\}\neq\{y_{3},y_{4}\}\Big\}.$$ We also introduce for $\bfx=(x_{1},\ldots, x_{2g+4})\in\mathbb{N}^{2g+4}$ the collection of sets
\begin{equation}\label{sx}\mathcal{S}_{\bfx}=\Big\{\{x_{1},x_{4},x_{2l-1},x_{2l}\},\ \ 3\leq l\leq g+2\Big\}\cup \Big\{\{x_{2j-1},x_{2j},x_{2l-1},x_{2l}\},\ \ 3\leq l<j\leq g+2\Big\},\end{equation} and for $\tilde{\bfx}=(x_{4},\ldots,x_{2g+4})\in\mathbb{N}^{2g+1}$ and $r_{1}\gg n^{\varepsilon}$ the set
\begin{equation}\label{sxs}\mathcal{S}_{\tilde{\bfx},r_{1}}=\Big\{\{r_{1},x_{4},x_{2l-1},x_{2l}\},\ \ 3\leq l\leq g+2\Big\}\cup \Big\{\{x_{2j-1},x_{2j},x_{2l-1},x_{2l}\},\ \ 3\leq l<j\leq g+2\Big\}.\end{equation}
The upcoming auxiliary technical lemma shall be employed henceforth, its proof having its reliance on the same circle of ideas.
\begin{lem}\label{lem7.4}
Whenever $\o\neq \mathcal{R}\subset{\mathbb{N}}\cap [1,n]$ is a fixed set and $r_{1}\in\mathbb{N}$ with $r_{1}\gg n^{\varepsilon}$ is fixed then there exists some constant $\delta_{g}>0$ for which
\begin{equation*}\sum_{\substack{\bfx\in T_{n}^{ind}\\ \{y_{1},y_{2},y_{3},y_{4}\}\in \mathcal{C}_{\bfr,\bfx}\setminus \mathcal{S}_{\bfx}}}\prod_{i=1}^{2g+4}x_{i}^{-\alpha}\ll n^{-\delta_{g}},\ \  \ \sum_{\substack{\bfx\in V_{r_{1}}\\ \{y_{1},y_{2},y_{3},y_{4}\}\in \mathcal{C}_{\bfr,\bfx}\setminus \mathcal{S}_{\bfx,r_{1}}}}\prod_{i=4}^{2g+4}x_{i}^{-\alpha}\ll n^{-\delta_{g}}.\end{equation*} 
\end{lem}
\begin{proof}

We begin by examining the first sum and assume that $\{y_{1},y_{2},y_{3},y_{4}\}\subset\bfr\cup\{x_{1},x_{4}\}$, an ensuing consequence being that either $x_{4}$ can be expressed in terms of $x_{1}$ and $\bfr$ or $x_{1}=f_{\bfr}$ is determined by $\bfr$. Both of the contributions would then be bounded above via an application of Lemma \ref{lem1.1} combined with (\ref{ela}) by
\begin{equation}\label{mario} \sum_{\substack{\{x_{1},x_{2},x_{3}\}\in R_{n}\\ }}x_{1}^{(1-2\alpha)g-\alpha}(x_{2}x_{3})^{-\alpha}+\sum_{\substack{\{x_{1},x_{2},x_{3}\}\in R_{n}\\ \text{$x_{1}=f_{\bfr}$}}}(x_{1}x_{2}x_{3})^{-\alpha}
\ll_{\lambda} (n^{(1-2\alpha)g\varepsilon}+n^{-\varepsilon})\log n,
\end{equation}
as desired. 

We assume without loss of generality that $y_{1}+y_{2}=y_{3}+y_{4}$ and suppose that $x_{2}\in \{y_{1},y_{2},y_{3},y_{4}\}$, say $x_{2}=y_{4}$, the instance for which $x_{3}\in \{y_{1},y_{2},y_{3},y_{4}\}$ being symmetrical. If $y_{3}\neq x_{3}$ then once $x_{1},x_{4},\ldots,x_{2g+4}$ are fixed, the variables $x_{2}$ and $x_{3}$ are completely determined. Such an observation enables one to bound the contribution corresponding to that case via a routine application of Lemma \ref{lem1.1} and (\ref{esti}), namely
\begin{align*}\sum_{\substack{\bfx\in T_{n}^{ind}}}\prod_{i=1}^{2g+4}x_{i}^{-\alpha}\sum_{\substack{\{y_{1},y_{2},y_{3},x_{2}\}\in\mathcal{C}_{\bfr,\bfx}\\  x_{3}\neq y_{3}}}1&\ll n^{-2\alpha-\varepsilon\alpha}\sum_{\substack{\bfx\in V_{x_{1}}\\ x_{4}<x_{1}<n}}\prod_{i=4}^{2g+4}x_{i}^{-\alpha}\ll n^{-\varepsilon \alpha+1-2\alpha}.
\end{align*}

Similarly, we suppose that $x_{2l+1}\in \{y_{1},y_{2},y_{3},y_{4}\}$ for some $2\leq l\leq g+1$, say $x_{2l+1}=y_{4}$ and $l=g+1$ without loss of generality. If $y_{3}\neq x_{2g+4}$ then once $x_{1},\ldots, x_{2g+2}$ are fixed, $x_{2g+3}$ and $x_{2g+4}$ are completely determined, it further being apparent that $x_{2g+3}x_{2g+4}\gg x_{1}$. We employ the preceding discussion in conjunction with (\ref{ela}) and Lemma \ref{lem1.1} to derive
\begin{align}\label{smile}\sum_{\substack{\bfx\in T_{n}^{ind}\\ \{y_{1},y_{2},y_{3},x_{2g+3}\}\in \mathcal{C}_{\bfr,\bfx}\\ x_{2g+4}\neq y_{3}}}\prod_{i=1}^{2g+4}x_{i}^{-\alpha}&\ll \sum_{\substack{\{x_{1},x_{2},x_{3}\}\in R_{n}\\ x_{4}\leq x_{1}}}x_{1}^{(1-2\alpha)(g-1)-2\alpha}(x_{2}x_{3}x_{4})^{-\alpha}\ll_{\lambda} n^{(1-2\alpha)g\varepsilon}\log n.
\end{align}
When on the contrary none of the above happens and $x_{2l+1}=y_{3}$ and $x_{2l+2}=y_{4}$ then $y_{1}+y_{2}=x_{1}+x_{4}$, and in view of the assumption $\{y_{1},y_{2},y_{3},y_{4}\}\notin\mathcal{S}_{\bfx}$ then either $\{y_{1},y_{2}\}\in\bfr$ or $\{y_{1},y_{2}\}=\{x_{2},x_{3}\}$. It then transpires that the argument leading to (\ref{mario}) combined with the identity $x_{1}+x_{2}+x_{3}=n$ would permit one to derive a similar conclusion in this instance. Likewise, if $$\{y_{1},y_{2},y_{3},y_{4}\}\cap \Big(\bigcup_{m=5}^{2g+4}\{x_{m}\}\Big)=\o$$ and $x_{2}=y_{4}$ and $x_{3}=y_{3}$ then $n-x_{1}=y_{1}+y_{2}$, wherein $y_{1},y_{2}\in\bfr\cup \{x_{1},x_{4}\}$. We estimate as well this contribution by means of the argument leading to (\ref{mario}).

For the second estimate in the statement we note first that a similar argument to the one displayed in the setting underlying (\ref{smile}) leads to the same conclusion. If one is not in the preceding context and $\{x_{2l+1},x_{2l+2}\}=\{y_{3},y_{4}\}$ then $y_{1},y_{2}\in\bfr$ as above, whence it transpires that $x_{4}$ would then be completely determined, the same conclusion holding if $\{y_{1},y_{2},y_{3},y_{4}\}\subset\bfr\cup\{r_{1},x_{4}\}$, such an observation in conjunction with analogous manoeuvres delivering the desired estimate.

\end{proof}

We shall include one last technical lemma in the spirit of those already exposed, it being worth recalling (\ref{prras}) and (\ref{rac1}), presenting for $\bfx\in\mathbb{N}^{2g+4}$ the set
\begin{equation}\label{pata}\mathcal{C}_{\bfr,\bfx}(n)=\Big\{\{y_{1},y_{2},y_{3}\}\in \mathcal{C}^{3}_{\bfr,\bfx}:\ \ \ \ \ \ \ \{y_{1},y_{2},y_{3}\}\in R_{n}\Big\}\end{equation}
and introducing for $\bfr\subset \mathbb{N}\cap [1,n]$, fixed $r_{1}\gg n^{\varepsilon}$ and $j=0,3$ the random variable
$$E_{\bfr}^{j}(n)=\sum_{\substack{\bfx\in \mathcal{T}_{j}}}P_{j}(\bfx)\sum_{\substack{\{y_{1},y_{2},y_{3}\}\in \mathcal{C}_{\bfr,\tilde{\bfx}}(n)\\ \{x_{1},x_{2},x_{3}\}\cap\{y_{1},y_{2},y_{3}\}=\o}}1,\ \   \text{wherein}\ P_{0}(\bfx)=\prod_{i=1}^{2g+4}x_{i}^{-\alpha},\ \ P_{3}(\bfx)=\prod_{i=4}^{2g+4}x_{i}^{-\alpha}$$
with  $\mathcal{T}_{0}=T_{n}^{ind}$, $\mathcal{T}_{3}=V_{r_{1}}$, the terms $x_{i},y_{i}$ for $1\leq i\leq 3$ denoting the first coordinates of $\bfx$ and $\bfy$ and \begin{equation}\label{rariti}\tilde{\bfx}=(x_{4},\ldots,x_{2g+4}).\end{equation}  

\begin{lem}\label{lem7.5} Let $\bfr\subset \mathbb{N}\cap [1,n]$. Then for some $\delta_{g}>0$ one has
 $$E_{\bfr}^{j}(n)\ll n^{-\delta_{g}}.\ \ \ \ \ \ \ \ \ \  j=0,3.$$
\end{lem}
\begin{proof}
We observe first that whenever $x_{2l+1}\in \{y_{1},y_{2},y_{3}\}$ for some $2\leq l\leq g+1$, say $l=g+1$, and $x_{2l+2}\notin \{y_{1},y_{2},y_{3}\}$ then once $x_{1},\ldots, x_{2g+2}$ are fixed, $x_{2g+3}$ and $x_{2g+4}$ are determined and $x_{2g+3}x_{2g+4}\gg x_{1}$. Under these circumstances then it would follow after subsequent applications of Lemma \ref{lem1.1} for the case $j=0$ that
\begin{align*}E_{\bfr}^{0}(n)&\ll \sum_{\substack{\{x_{1},x_{2},x_{3}\}\in R_{n}\\ x_{1}+x_{4}=\ldots=x_{2g+1}+x_{2g+2}\\ }}x_{1}^{-2\alpha}(x_{2}x_{3}\cdots x_{2g+2})^{-\alpha}
\\
&\ll \sum_{\substack{\{x_{1},x_{2},x_{3}\}\in R_{n} }}(x_{2}x_{3})^{-\alpha}x_{1}^{(1-2\alpha)(g-1)-2\alpha}\sum_{x_{4}\leq x_{1}}x_{4}^{-\alpha}\ll_{\lambda} (\log n)n^{(1-2\alpha)g\varepsilon},
\end{align*}
where in the last step we applied (\ref{ela}). For the case $j=3$ then the corresponding contribution in the above situation would be bounded with the aid of Lemma \ref{lem1.1} by
\begin{align*}E_{\bfr}^{3}(n)\ll \sum_{\substack{ r_{1}+x_{4}=\ldots=x_{2g+1}+x_{2g+2}\\ x_{4}\leq r_{1}}}r_{1}^{-\alpha}\prod_{i=4}^{2g+2}x_{i}^{-\alpha}&\ll r_{1}^{-\alpha+(1-2\alpha)(g-1)}\sum_{\substack{ x_{4}\leq r_{1} }}x_{4}^{-\alpha}\ll n^{\varepsilon(1-2\alpha)g}.
\end{align*}
If on the contrary the above does not happen and $x_{2l+1},x_{2l+2}\in \{y_{1},y_{2},y_{3}\}$ then one would have $x_{1}+x_{4}+y_{3}=n$ with $y_{3}\in\bfr$ or $y_{3}=x_{4}$, the ensuing conclusion being that $x_{4}$ is determined once $x_{1}$ is fixed and $x_{2l+1}x_{2l+2}\gg x_{1}\gg n^{\varepsilon}$. The same argument that leads to (\ref{mario}) permits one to obtain that the corresponding contribution is $O(n^{-\varepsilon\alpha})$. The remaining situation is that in which $y_{1}=x_{4}$ and $y_{2},y_{3}\in\bfr$, it being analogous to the previously described one.
\end{proof}
Before examining the expected value of $X_{n,\bfr}(A)$ it seems worth recalling (\ref{pata}) and presenting the following auxiliary technical lemma.
\begin{lem}\label{lem7.6}
Let $\bfr\subset \mathbb{N}\cap [1,n]$ with $\lvert \bfr\rvert\leq 2g+4$ and fixed $r_{1}\gg n^{\varepsilon}$. For some $\delta_{g}>0$ then
$$\sum_{\substack{\bfx\in T_{n}^{ind}}}\prod_{i=1}^{2g+4}x_{i}^{-\alpha}\sum_{\substack{\{y_{1},y_{2},y_{3}\}\in \mathcal{C}_{\bfr,\bfx}(n)\\ \{y_{1},y_{2},y_{3}\}\neq\{x_{1},x_{2},x_{3}\}}}1\ll n^{-\delta_{g}},\ \ \ \ \ \sum_{\substack{\bfx\in V_{r_{1}}\\ \{y_{1},y_{2},y_{3}\}\in \mathcal{C}_{\bfr,\bfx}(n)}}\prod_{i=4}^{2g+4}x_{i}^{-\alpha}\ll n^{-\delta_{g}}.$$
\end{lem}
\begin{proof}
We begin by assuming firstly that $\lvert\{y_{1},y_{2},y_{3}\}\cap\{x_{1},x_{2},x_{3}\}\rvert=1$ and note 
\begin{equation}\label{kalis}\sum_{\substack{\bfx\in T_{n}^{ind}\\ \{y_{1},y_{2},y_{3}\}\in \mathcal{C}_{\bfr,\bfx}(n)\\ \lvert\{y_{1},y_{2},y_{3}\}\cap\{x_{1},x_{2},x_{3}\}\rvert=1}}\prod_{i=1}^{2g+4}x_{i}^{-\alpha}\ll \sum_{\substack{\bfx\in T_{n}^{ind}\\  \{y_{1},y_{3},x_{1}\}\in \mathcal{C}_{\bfr,\bfx}(n)\\  \{y_{1},y_{3}\}\neq \{x_{2},x_{3}\}}}\prod_{i=1}^{2g+4}x_{i}^{-\alpha}+\sum_{\substack{\bfx\in T_{n}^{ind}\\ \{y_{1},y_{3},x_{2}\}\in \mathcal{C}_{\bfr,\bfx}(n)\\  \{y_{1},y_{3}\}\neq \{x_{1},x_{3}\}}}\prod_{i=1}^{2g+4}x_{i}^{-\alpha}.
\end{equation}
We analyse the second summand in the above equation, recall (\ref{rariti}) and observe that 
\begin{align*}n^{2\alpha+\varepsilon\alpha}\sum_{\substack{\bfx\in T_{n}^{ind}\\ \{y_{1},y_{3},x_{2}\}\in \mathcal{C}_{\bfr,\bfx}(n)\\ \{y_{1},y_{3}\}\neq \{x_{1},x_{3}\}}}\prod_{i=1}^{2g+4}x_{i}^{-\alpha}&\ll \sum_{\substack{\tilde{\bfx}\in V_{x_{1}}\\ x_{4}<x_{1}\leq n}}\prod_{i=4}^{2g+4}x_{i}^{-\alpha}\sum_{\substack{x_{2}+x_{3}=n-x_{1}\\ x_{2}+y_{1}+y_{3}=n\\ y_{1},y_{3}\in\text{Set}(\tilde{\bfx})\cup \bfr}}1\ll \sum_{x_{1}\leq n}1\ll n,
\end{align*}
where we employed Lemma \ref{lem1.1} and (\ref{esti}) in the second step. For the first summand in (\ref{kalis}) we discard first the instance when $y_{1},y_{3}\in\bfr$ by deriving in a similar fashion
\begin{align*}n^{2\alpha+\varepsilon\alpha}\sum_{\substack{\bfx\in T_{n}^{ind}\\ \{y_{1},y_{3},x_{1}\}\in \mathcal{C}_{\bfr,\bfx}(n)\\ y_{1},y_{3}\in\bfr}}\prod_{i=1}^{2g+4}x_{i}^{-\alpha}&\ll \sum_{\substack{x_{1}+x_{3}=n-x_{2}\\ x_{1}+y_{1}+y_{3}=n\\ y_{1},y_{3}\in\bfr}}\ \ \sum_{\substack{\tilde{\bfx}\in V_{x_{1}}\\ x_{4}<x_{1}}}\prod_{i=4}^{2g+4}x_{i}^{-\alpha}\ll \sum_{x_{2}\leq n}1\ll n.
\end{align*} 

If $y_{1}=x_{2i+1}$ for $2\leq i\leq g+1$ (and similarly for even index) and $y_{3}\neq x_{2i+2}$, say $i=g+2$ without loss of generality, then the corresponding contribution will be bounded above by
\begin{align*} &\sum_{\substack{\{x_{1},x_{2},x_{3}\}\in R_{n}\\ x_{1}+x_{4}=\ldots=x_{2g+1}+x_{2g+2}\\ x_{4}\leq x_{1}}}(x_{1}x_{2}x_{3})^{-\alpha}\prod_{j=4}^{2g+2}x_{j}^{-\alpha} \sum_{\substack{\{x_{1},x_{2g+3},y_{3}\}\in  \mathcal{C}_{\bfr,\bfx}(n) \\ x_{1}+x_{4}=x_{2g+3}+x_{2g+4}\\ y_{3}\neq x_{2i+2}}}(x_{2g+3}x_{2g+4})^{-\alpha}
\\
&\ll \sum_{\substack{\{x_{1},x_{2},x_{3}\}\in R_{n}}}x_{1}^{(1-2\alpha)(g-1)-2\alpha}(x_{2}x_{3})^{-\alpha}\sum_{x_{4}\leq x_{1}}x_{4}^{-\alpha}\ll_{\lambda} n^{(1-2\alpha)g\varepsilon}\log n,
\end{align*}
wherein we applied as is customary Lemma \ref{lem1.1} and (\ref{ela}). When instead $y_{1}=x_{2i+1}$ and $y_{3}=x_{2i+2}$ then such a contribution shall be bounded above by
\begin{align*}\sum_{\substack{\bfx\in T_{n}^{ind}\\ x_{4}\leq x_{1}\\ 2x_{1}+x_{4}=n}}\prod_{i=1}^{2g+4}x_{i}^{-\alpha}&\ll \sum_{\substack{\{x_{1},x_{2},x_{3}\}\in R_{n}}}(x_{1}x_{2}x_{3})^{-\alpha}x_{1}^{(1-2\alpha)g}\ll_{\lambda} n^{\varepsilon(1-2\alpha)g}\log n,
\end{align*}
wherein we employed (\ref{ela}). The instance when $\{y_{1},y_{3}\}\in \{x_{4}\}\cup \bfr$ shall be done in an analogous manner. If on the contrary $\{y_{1},y_{2},y_{3}\}\cap\{x_{1},x_{2},x_{3}\}=\o$ then an application of Lemma \ref{lem7.5} for the case $j=0$ would suffice. The preceding discussion permits one to derive the first estimate in the lemma. The second one follows by a similar argument combined with an application of Lemma \ref{lem7.5} for the case $j=3$.
\end{proof}
Equipped with the above lemmata and recalling (\ref{tama}) we derive the following result.
\begin{prop}\label{propo7}
Whenever $\bfr\subset\mathbb{N}\cap [1,n]$ satisfies $\lvert \bfr\rvert\leq 2g+3$ one has that
\begin{equation*}\mathbb{E}\big(\lvert X_{n,\bfr}(A)\rvert\ \big|\bfr\subset A\big)\ll n^{-\delta_{g}},\ \ \ \ \ \ \ \ \ \end{equation*} wherein $\delta_{g}>0$ is fixed and the implicit constant does not depend on $\bfr$.
\end{prop}
\begin{proof}
An insightful inspection of the definition (\ref{tama}) reveals that 
\begin{align}\label{kiruna}\mathbb{E}\big(\lvert X_{n,\bfr}(A)\rvert\ \big|\bfr\subset A\big)\ll H_{0,\bfr}(n)+H_{1,\bfr}(n)+H_{3,\bfr}(n),
\end{align}
where
\begin{equation}\label{aas}H_{i,\bfr}(n)=\sum_{\substack{\bfx\in T_{n}^{ind}\\ \o\neq \overline{\bfx}\subset\text{Set}(\bfx)\\ \text{Set}(\bfx)\cap\bfr=\o}}\mathbb{P}\big(\text{Set}(\bfx)\subset A\big)\mathop{{\sum_{\substack{\bfy\in  T_{n,\bfr\cup \overline{\bfx}}^{5}\\ |\{x_{1},x_{2},x_{3}\}\cap\{y_{1},y_{2},y_{3}\}|=i\\  }}}^*}\mathbb{P}\big(\text{Set}(\bfy)\subset A \big|\ \overline{\bfx},\bfr\subset A\big)\end{equation}
for $i=0,1,3,$ and wherein whenever $\bfr=\o$ then the underlying tuples $\bfy$ don't satisfy that either $\text{Set}(\bfx)\cap \text{Set}(\bfy)=\{x_{1},x_{2},x_{3}\}=\{y_{1},y_{2},y_{3}\}$ with $\bfy\in T_{n}^{ind}$, or $\text{Set}(\bfx)=\text{Set}(\bfy)$. We examine first the case $i=3$, apply Lemma \ref{lem1.1} and split the corresponding sum into
\begin{align*}
H_{3,\bfr}(n)\ll&\sum_{j=1}^{3}\sum_{\bfx\in T_{n}^{ind}}(x_{1}\cdots x_{2g+4})^{-\alpha}\Bigg(\mathop{{\sum_{\substack{x_{j}+y_{4}=y_{5}+r_{6}\\ r_{6}\in\bfr\cup\text{Set}(\bfx)}}}^*}(y_{4}y_{5})^{-\alpha}+\mathop{{\sum_{\substack{x_{j}+r_{4}=r_{5}+y_{6}\\ r_{4},r_{5}\in\bfr\cup\text{Set}(\bfx)}}}^*}y_{6}^{-\alpha}
\\
&+\mathop{{\sum_{\substack{x_{j}+y_{4}=r_{5}+r_{6}\\ r_{5},r_{6}\in\bfr\cup\text{Set}(\bfx)}}}^*}y_{4}^{-\alpha}+\mathop{{\sum_{\substack{x_{j}+r_{4}=r_{5}+r_{6}\\ r_{4},r_{5},r_{6}\in\bfr\cup\text{Set}(\bfx)}}}^*}1\Bigg),
\end{align*}
where $\max(r_{5},r_{6})< x_{j}$ and in the last summand $\{x_{j},r_{4},r_{5},r_{6}\}\notin \mathcal{S}_{\bfx}$, it being worth recalling (\ref{sx}). We observe that Lemma \ref{lem1.1} yields
\begin{align}\label{H3}H_{3,\bfr}(n)\ll &\sum_{j=1}^{3}\mathop{{\sum_{\substack{\bfx\in T_{n}^{ind}\\ r_{4},r_{5},r_{6}\in \bfr\cup\text{Set}(\bfx)}}}^*}\prod_{i=1}^{2g+4}x_{i}^{-\alpha}\Big((x_{j}-r_{6})^{1-2\alpha}+(x_{j}+r_{4}-r_{5})^{-\alpha}
\\
&+(r_{5}+r_{6}-x_{j})^{-\alpha}+\leavevmode\hbox{$1\!\rm I$}_{r_{5}+r_{6}-r_{4}}(x_{j})\Bigg),\nonumber
\end{align}
wherein the tuples satisfy the additional property that the expressions in brackets are positive and in the last summand $\{x_{j},r_{4},r_{5},r_{6}\}\notin \mathcal{S}_{\bfx}$. The use of Lemma \ref{lem32} permits one to deduce that the contribution corresponding to the first three summands in the above equation is $O(n^{-\delta_{g}})$ for some $\delta_{g}>0$, the one corresponding to the last summand being estimated via the employment of Lemma \ref{lem7.4}.

We shall focus our attention next on the term $H_{1,\bfr}(n)$ and note that an application of Lemmata \ref{lem7.2} and \ref{lem7.6}  to (\ref{aas}) reveals that 
\begin{align*}H_{1,\bfr}(n)&\ll \sum_{\substack{\bfx\in T_{n}^{ind}}}\prod_{i=1}^{2g+4}x_{i}^{-\alpha}\sum_{\substack{\{y_{1},y_{2},y_{3}\}\in \mathcal{C}_{\bfr,\bfx}(n)\\   \lvert\{y_{1},y_{2},y_{3}\}\cap\{x_{1},x_{2},x_{3}\}\rvert=1}}1\ll n^{-\delta_{g}},
\end{align*}
as desired. We conclude our investigation by examining $H_{0,\bfr}(n)$ and write
\begin{align}\label{kalim}H_{0,\bfr}(n)\ll & \sum_{\substack{\bfx\in T_{n}^{ind}\\ \text{Set}(\bfx)\cap\bfr=\o}}\mathbb{P}\big(\text{Set}(\bfx)\subset A\big)\mathop{{\sum_{\substack{\bfy\in T_{n}\\ \{y_{1},y_{2},y_{3}\}\subset \bfr\subset\text{Set}(\bfy)\\ \text{Set}(\bfx)\cap \text{Set}(\bfy)\neq \o}}}^*}\mathbb{P}\big(\text{Set}(\bfy)\subset A \big|\ \text{Set}(\bfx),\bfr\subset A\big)
\\
&+ \sum_{\substack{\bfx\in T_{n}^{ind}\\ \text{Set}(\bfx)\cap\bfr=\o}}\mathbb{P}\big(\text{Set}(\bfx)\subset A\big)\mathop{{\sum_{\substack{\bfy\in T_{n}\\ \bfr\subset\text{Set}(\bfy)\\ \{y_{1},y_{2},y_{3}\}\in\mathcal{C}_{\bfx,\bfr}(n)}}}^*}\mathbb{P}\big(\text{Set}(\bfy)\subset A \big|\text{Set}(\bfx),\bfr\subset A\big),\nonumber
\end{align}
wherein we omitted as we shall do henceforth writing $\{y_{1},y_{2},y_{3}\}\cap\{x_{1},x_{2},x_{3}\}=\o$. We employ a similar argument as the one to derive (\ref{H3}) to deduce that the first summand in the above equation is bounded above by a constant times
\begin{align*}\sum_{\substack{\bfx\in T_{n}^{ind}}}\prod_{i=1}^{2g+4}x_{i}^{-\alpha}\sum_{\substack{r_{4},r_{5},r_{6}\in \bfr\cup\text{Set}(\bfx)\\ \overline{x}\in\text{Set}(\bfx)}}\Bigg(&(r-\overline{x})^{1-2\alpha}+(r+r_{4}-\overline{x})^{-\alpha}+(r+\overline{x}-r_{5})^{-\alpha}
\\
&+(\overline{x}+r_{6}-r)^{-\alpha}+\leavevmode\hbox{$1\!\rm I$}_{r_{5}+r_{6}-r}(\overline{x})+\leavevmode\hbox{$1\!\rm I$}_{r+r_{4}-r_{5}}(\overline{x})\Bigg)
\end{align*}
for some $r\in\bfr$, where we omitted indicating that each linear factor in the above lines is positive. We note as is customary that an application of Lemmata \ref{lem32} and \ref{lem7.4} enables one to deduce that the contribution corresponding to the first four summands and the last two respectively in the preceding lines is $O(n^{-\delta_{g}})$ for some $\delta_{g}>0$. An application of Lemma \ref{lem7.2} and \ref{lem7.5} for the instance $j=0$ would then permit one to deduce that the second summand in (\ref{kalim}) is $O(n^{-\delta_{g}}).$ The combination of the corresponding bounds for $H_{i,\bfr}(n)$ and (\ref{kiruna}) enables one to derive the desired estimate for $\mathbb{E}\big(\lvert X_{n,\bfr}(A)\rvert\ \big|\bfr\subset A\big)$.
\end{proof}

We continue by taking a fixed $r_{1}\gg n^{\varepsilon}$, recalling (\ref{prras}), (\ref{Tnr3n}), (\ref{Vn}), defining for $\bfr\subset \mathbb{N}$ the function $\leavevmode\hbox{$1\!\rm I$}_{\o,3}(\bfr)=1$ if $r_{1}\in \bfr\in R_{n}$ and $\leavevmode\hbox{$1\!\rm I$}_{\o,3}(\bfr)=0$ else and introduce the sum
\begin{equation}\label{Bruja}Y_{n,\bfr}(A)=\sum_{\substack{\bfx\in V_{r_{1}}\\ \text{Set}(\bfx)\cap\bfr=\o}}\prod_{j=4}^{2g+4}\leavevmode\hbox{$1\!\rm I$}_{A}(x_{j})\Big(-\leavevmode\hbox{$1\!\rm I$}_{\o,3}(\bfr)+\sum_{\substack{\o\neq \overline{\bfx}\subset\text{Set}(\bfx)}}\lvert T_{n,\bfr\cup \overline{\bfx}}^{5}(A)\rvert\Big).\end{equation}

\begin{prop}\label{cila}
One has for $\bfr\subset \mathbb{N}\cap [1,n]$ with $\lvert \bfr\rvert \leq 2g+3$ and $r_{1}\gg n^{\varepsilon}$ that
\begin{equation*}\mathbb{E}\big(\lvert Y_{n,\bfr}(A)\rvert\ \big|\bfr\subset A\big)\ll n^{-\delta_{g}}, \end{equation*} wherein $\delta_{g}>0$ is fixed and the implicit constant does not depend on $\bfr$.
\end{prop}
\begin{proof}
We observe as in the discussion concerning $H_{0,\bfr}(n)$ in Proposition \ref{propo7} that
\begin{align*}\label{kalim}\mathbb{E}\big(\lvert Y_{n,\bfr}(A)&\rvert\ \big|\bfr\subset A\big)\ll  \sum_{\substack{\bfx\in V_{r_{1}}\\  \o\neq \overline{\bfx}\subset \text{Set}(\bfx)\\ \text{Set}(\bfx)\cap \bfr=\o}}\prod_{i=4}^{2g+4}x_{i}^{-\alpha}\mathop{{\sum_{\substack{\bfy\in T_{n}\\ \overline{\bfx},\bfr\subset\text{Set}(\bfy)\\ \{y_{1},y_{2},y_{3}\}\subset \bfr}}}^*}\mathbb{P}\big(\text{Set}(\bfy)\subset A \big|\ \overline{\bfx},\bfr\subset A\big)
\\
&+ \sum_{\substack{\bfx\in V_{r_{1}}}}\prod_{i=4}^{2g+4}x_{i}^{-\alpha}\sum_{\substack{\bfy\in T_{n}\\ \bfr\subset\text{Set}(\bfy)\\ \{y_{1},y_{2},y_{3}\}\in C_{\bfr,\bfx}(n)}}\mathbb{P}\big(\text{Set}(\bfy)\subset A \big|\ \text{Set}(\bfx),\bfr\subset A\big),\nonumber
\end{align*}where in the first sum $\text{Set}(\bfx)\not\subset \text{Set}(\bfy)$ whenever $r_{1}\in \bfr\in R_{n}$. Upon recalling (\ref{sxs}) then the argument in (\ref{H3}) permits one to bound the first summand in the above line by a constant times
\begin{align*}\sum_{\substack{\bfx\in V_{r_{1}}\\ }}\prod_{i=4}^{2g+4}&x_{i}^{-\alpha}\mathop{{\sum_{\substack{r_{1}',r_{4},r_{5},r_{6}\in \bfr\cup\text{Set}(\bfx)\\ \overline{x}\in\text{Set}(\bfx)}}}^*}\Bigg((r_{1}'-\overline{x})^{1-2\alpha}+(r_{1}'+r_{4}-\overline{x})^{-\alpha}+(r_{1}'+\overline{x}-r_{5})^{-\alpha}
\\
&+(r_{5}+\overline{x}-r_{1}')^{-\alpha}+\leavevmode\hbox{$1\!\rm I$}_{r_{5}+r_{6}-r_{1}'}(\overline{x})+\leavevmode\hbox{$1\!\rm I$}_{r_{1}'+r_{4}-r_{5}}(\overline{x})\Bigg),
\end{align*}
where in the above sum $\{r_{1}',\overline{x},r_{5},r_{6}\},\{r_{1}',r_{4},r_{5},\overline{x}\}\notin \mathcal{S}_{\bfx,r_{1}}$, an application of Lemmata \ref{lem32} and \ref{lem7.4} enabling one to derive that the preceding sum is $O(n^{-\delta_{g}})$. Similarly, we may deduce via Lemma \ref{lem7.2} that \begin{align*}\sum_{\substack{\bfx\in V_{r_{1}}}}&\prod_{i=4}^{2g+4}x_{i}^{-\alpha}\sum_{\substack{\bfy\in T_{n}\\ \bfr\subset\text{Set}(\bfy)\\ \{y_{1},y_{2},y_{3}\}\in \mathcal{C}_{\bfr,\bfx}(n)}}\mathbb{P}\big(\text{Set}(\bfy)\subset A \big|\ \text{Set}(\bfx),\bfr\subset A\big)\ll \sum_{\substack{\bfx\in V_{r_{1}}\\ \{y_{1},y_{2},y_{3}\}\in \mathcal{C}_{\bfr,\bfx}(n)}}\prod_{i=4}^{2g+4}x_{i}^{-\alpha},
\end{align*}
whence the use of Lemma \ref{lem7.6} allows one to deduce that the above line is $O( n^{-\delta_{g}})$. Combining the previous estimates we derive the desired result.
\end{proof}

We shall devote the last lines of this section to deliver suitable estimates for the expected value of yet another auxiliary random variable. For such purposes we recall (\ref{Tnre3}) and introduce for $\overline{\bfx},\bfr\subset\mathbb{N}$ with $\lvert \bfr\rvert\leq 2g+3$ the sets 
$$\mathcal{Z}^{1}_{n,\bfr,\overline{\bfx}}=\bigcup_{j=2}^{g+1}\Big\{\bfy\in T^{5}_{n,\bfr\cup \overline{\bfx}}:\ \ \  \{y_{2j+1},y_{2j+2}\}\cap\overline{\bfx}\neq \o ,\ \ \ \{y_{2j+1},y_{2j+2}\}\not\subset \bfr\cup  \overline{\bfx} \Big\},$$ and upon recalling (\ref{pata}) then $$\mathcal{Z}^{2}_{n,\bfr,\overline{\bfx}}=\Big\{\bfy\in T^{5}_{n,\bfr\cup \overline{\bfx}}:\ \ \  \{y_{1},y_{2},y_{3}\}\in \mathcal{C}_{\bfr,\overline{\bfx}}(n)\Big\},$$ and
\begin{equation}\label{NNN}\mathcal{Z}_{n,\bfr,\overline{\bfx}}=\mathcal{Z}^{1}_{n,\bfr,\overline{\bfx}}\cup\mathcal{Z}^{2}_{n,\bfr,\overline{\bfx}},\ \ \ \ \ \ \ \ \ \ \ \ \mathcal{Z}_{n,\bfr,\overline{\bfx}}(A)=\Big\{\bfy \in \mathcal{Z}_{n,\bfr,\overline{\bfx}}:\ \ \ \ \text{Set}(\bfy)\subset A\big\}\end{equation} for a sequence $A\subset\mathbb{N}$. Likewise, we introduce
\begin{equation}\label{Defi}\mathcal{Z'}_{n,\bfr,\overline{\bfx}}=\Big\{\bfy\in T^{5}_{n,\bfr\cup \overline{\bfx}}:\ \ \ \ \ \{y_{2j+1},y_{2j+2}\}\subset\overline{\bfx}\ \ \ \text{for some $2\leq j\leq g+1$} \Big\}.\end{equation}
\begin{lem}\label{cilas}
For $\bfr,\overline{\bfr}\subset\mathbb{N}\cap [1,n]$ with $\lvert \bfr\rvert\leq 2g+3$ and $g+2\leq \lvert\overline{\bfr}\rvert\leq 2g+3$ one may partition the set $T^{5}_{n,\overline{\bfr}}$ into disjoint sets $\mathcal{W}_{1,\overline{\bfr}}$ and $\mathcal{W}_{2,\overline{\bfr}}$ and such that $$\sum_{\substack{\bfx\in \mathcal{W}_{1,\overline{\bfr}}\\ \text{Set}(\bfx)=\bfx'\cup\overline{\bfr}\\ \bfx'\cap\overline{\bfr}=\o}}\mathbb{P}(\bfx'\subset A)\ll n^{-\delta_{g}}\ \ \ \ \ \ \ \ \text{and}\ \ \ \ \ \ \ \ \max_{\substack{\o\neq\overline{\bfx}\subset\text{Set}(\bfx)\\ \bfx\in\mathcal{W}_{2,\overline{\bfr}}}}\mathbb{E}\big(\lvert \mathcal{Z}_{n,\bfr,\overline{\bfx}}(A) \rvert \big|\bfr,\overline{\bfx}\subset A\big)\ll n^{-\delta_{g}}.$$ Likewise, when $\overline{\bfr'}\subset\mathbb{N}$ and $g+2\leq \lvert\overline{\bfr'}\rvert\leq 2g+3$ and for given $\bfx\in T_{n,\overline{\bfr}}^{5}$ then the tuples $\bfy\in T_{n,\overline{\bfr'}}^{5}$ with $\bfx'\cap \bfy'=\o$, where $\bfx',\bfy'$ are defined as above, have the property that there is some $\bfr_{1}\subset \mathbb{N}$ and some $\overline{\bfx}\subset \bfx'\cup \bfy'$ for which $\bfr_{1}\cap \overline{\bfx}=\o$ and $\mathcal{Z'}_{n,\bfr_{1},\overline{\bfx}}\neq \o$ one has that either
\begin{equation}\label{ggg}\sum_{\substack{\bfx\in T_{n,\overline{\bfr}}^{5}\\ \text{Set}(\bfx)=\bfx'\cup\overline{\bfr}\\ \bfx'\cap\overline{\bfr}=\o}}\mathbb{P}(\bfx'\subset A)\ll n^{-\delta_{g}}\  \ \ \ \ \ \text{or}\ \ \ \ \ \ \ \ \mathop{{\sum_{\substack{\bfy\in T_{n,\overline{\bfr'}}^{5}\\ \text{Set}(\bfy)=\bfy'\cup\overline{\bfr'}\\ \bfy'\cap\overline{\bfr'}=\o}}}^*}\mathbb{P}(\bfy'\subset A)\ll n^{-\delta_{g}},\end{equation} where in the second sum the tuples $\bfy$ satisfy the earlier described conditions. 
\end{lem}
\begin{proof}
We first observe that whenever $g+2\leq\lvert\overline{\bfr}\rvert\leq g+3$ then if $\bfx\in T_{n,\overline{\bfr}}^{5}$ one may assume in view of (\ref{Tnr3n}) and the discussion after it that $\bfx=(r_{1},r_{2},r_{3},x_{4},\ldots,r_{2g+3},x_{2g+4})$ with \begin{equation}\label{ben}r_{1}+x_{4}=r_{5}+x_{6}=\ldots=r_{2g+3}+x_{2g+4},\ \ \ \ \ \ \ \ \ \ \ r_{1}\geq r_{2l+1},\ \ \ 2\leq l\leq g+1,\end{equation} and $r_{i}\in \overline{\bfr},$ the case $\lvert\overline{\bfr}\rvert=g+2$ encompassing the instance when $r_{k}=r_{2f+1}$ for some $1\leq k\leq 3$ and $2\leq f\leq g+1$. We then note for fixed $0<\beta_{1}<\beta<\varepsilon$ that whenever $r_{1}-r_{2l+1}\geq n^{\beta_{1}}$ for some $2\leq l\leq g+1$ then an application of Lemma \ref{lem1.1} permits one to deduce that\begin{equation*}\sum_{\substack{\bfx\in T_{n,\overline{\bfr}}^{5}\\ \text{Set}(\bfx)=\bfx'\cup\overline{\bfr}\\ \bfx'\cap\overline{\bfr}=\o}}\mathbb{P}(\bfx'\subset A)\ll \mathop{{\sum_{x_{4},\ldots,x_{2g+4}}}^*}\Big(\prod_{j=2}^{g+2}x_{2j}^{-\alpha}\Big)\ll (r_{1}-r_{2l+1})^{(1-2\alpha)}\ll n^{(1-2\alpha)\beta_{1}},\end{equation*} wherein the above sum runs over tuples satisfying (\ref{ben}), in which case it would transpire that the choice $ \mathcal{W}_{1,\overline{\bfr}}=T_{n,\overline{\bfr}}^{5}$ would yield the lemma. If the above does not hold then by taking \begin{equation}\label{choice}\mathcal{W}_{1,\overline{\bfr}}=\Big\{\bfx \in T_{n,\overline{\bfr}}^{5}:\ \ \ \ x_{i}> n^{\beta}\ \text{for some $x_{i}\in\text{Set}(\bfx)\setminus \overline{\bfr}$} \Big\},\ \ \ \ \ \ \mathcal{W}_{2,\overline{\bfr}}=\mathcal{W}_{n,\overline{\bfr}}\setminus \mathcal{W}_{1,\overline{\bfr}},\end{equation} it would then follow after subsequent applications of Lemma \ref{lem1.1} that
\begin{align}\label{raul}\sum_{\substack{\bfx\in \mathcal{W}_{1,\overline{\bfr}}\\ \text{Set}(\bfx)=\bfx'\cup\overline{\bfr}\\ \bfx'\cap\overline{\bfr}=\o}}\mathbb{P}(\bfx'\subset A)\ll \mathop{{\sum_{\substack{x_{4},\ldots,x_{2g+4}\\ x_{2l}>n^{\beta}}}}^*}\Big(\prod_{j=2}^{g+2}x_{2j}^{-\alpha}\Big)&\ll  \sum_{x_{2l}> n^{\beta}}x_{2l}^{-\alpha}\big(x_{2l}-(r_{1}-r_{2l+1})\big)^{-\alpha}\nonumber
\\
&\ll  \sum_{x_{2l}> n^{\beta}}x_{2l}^{-2\alpha}\ll n^{(1-2\alpha)\beta}.\end{align}
Moreover, if $\overline{\bfx}\subset \text{Set}(\bfx)$ for $\bfx\in\mathcal{W}_{2,\overline{\bfr}}$ then it is apparent by Lemma \ref{lem1.1} that 
$$\mathbb{E}\big( \lvert \mathcal{Z}_{n,\bfr,\overline{\bfx}}(A)\rvert \big|\bfr,\overline{\bfx}\subset A\big)\ll (r_{1}'-\overline{x})^{-\alpha}+\sum_{\substack{r_{1}'+y_{4}=y_{5}+\overline{x}\\ \overline{x}\in\overline{\bfx}}}(y_{4}y_{5})^{-\alpha}\ll (r_{1}'-\overline{x})^{1-2\alpha}$$ for some $r_{1}'\in\bfr$, wherein we utilised the fact for the implicit triple $\{r_{1}',r_{2}',r_{3}'\}\in R_{n}$ that $\overline{x}\notin \{r_{1}',r_{2}',r_{3}'\}$ since $r_{1}',r_{2}',r_{3}'\gg n^{\varepsilon}$ but $\overline{x}\leq n^{\beta}.$ It therefore transpires that 
\begin{equation}\label{asp}\mathbb{E}\big( \lvert\mathcal{Z}_{n,\bfr,\overline{\bfx}}(A)\rvert \big|\bfr,\overline{\bfx}\subset A\big)\ll n^{-\delta_{g}}.\end{equation}

If on the contrary $g+4\leq\lvert\overline{\bfr}\rvert\leq 2g+3$ or $\lvert\overline{\bfr}\rvert=g+3$ but $\bfx$ is not as in (\ref{ben}) then in view of the definition of $T_{n,\overline{\bfr}}^{5}$ in (\ref{Tnr3n}) one would have $\lvert T_{n,\overline{\bfr}}^{5}\rvert=1$. The choice (\ref{choice}) then would deliver the estimate 
\begin{equation}\label{alp}\sum_{\substack{\bfx\in \mathcal{W}_{1,\overline{\bfr}}\\ \text{Set}(\bfx)=\bfx'\cup\overline{\bfr}\\ \bfx'\cap\overline{\bfr}=\o}}\mathbb{P}(\bfx'\subset A)\ll n^{-\beta\alpha},\end{equation} and the same argument as above would yield (\ref{asp}).

For the second estimate in the lemma it just suffices to observe that (\ref{Defi}) entails the existence of some $2\leq j\leq g+1$ for which $z_{2j+1}+z_{2j+2}\gg n^{\varepsilon}$ with $z_{2j+1},z_{2j+2}\in\overline{\bfx}$, wherein the previous summands are components of the underlying vector $\bfz$ counted by $\mathcal{Z}'_{n,\bfr_{1},\overline{\bfx}}$, whence in view of the fact that $\overline{\bfx}\subset\bfx'\cup\bfy'$ it is then apparent that there is some $x_{i}$ or $y_{i}$ in the underlying sums for (\ref{ggg}), say $x_{2g+4}$ without loss of generality, for which $x_{2g+4}\gg n^{\varepsilon}.$ The same argument as in (\ref{raul}) and (\ref{alp}) in conjunction with subsequent applications of Lemma \ref{lem1.1} enables one to derive (\ref{ggg}), as desired.
\end{proof}

\section{The upper tail}\label{sec9}
We shall devote the present section to derive a concentration type result to estimate the probability that $\lvert T_{n}(A)\rvert$ significantly exceeds its mean value via the analysis of a suitable moment. For such purposes we take $m=\big\lceil \log n\big\rceil$ and consider for further convenience
\begin{equation*}\lvert T_{n}(A)\rvert^{m}=\sum_{\bfx_{1}\in T_{n}}\leavevmode\hbox{$1\!\rm I$}_{A}(\bfx_{1})\sum_{\bfx_{2}\in T_{n}}\leavevmode\hbox{$1\!\rm I$}_{A}(\bfx_{2})\dots \sum_{\bfx_{m}\in T_{n}}\leavevmode\hbox{$1\!\rm I$}_{A}(\bfx_{m}),\end{equation*} whence taking expectations then delivers
\begin{equation}\label{EET}\mathbb{E}\big(\lvert T_{n}(A)\rvert^{m}\big)=\sum_{\bfx_{1}\in T_{n}}\sum_{\bfx_{2}\in T_{n}}\dots \sum_{\bfx_{i}\in T_{n}}\dots\sum_{\bfx_{m}\in T_{n}}\mathbb{P}(\bfx_{1},\ldots,\bfx_{m}\in A).\end{equation}
It seems then worth introducing $$\mathcal{B}=\{B\subset [1,m]\times [1,2g+4]:\ \ \ \ \ \ \lvert B\rvert\leq 2g+4\}$$ and taking the set of functions \begin{equation}\label{phi}\Phi=\Big\{\varphi:[1,m]\rightarrow  \mathcal{B},\ \ \ \ \ \ \ (v_{1},v_{2})\in\varphi(i)\Longrightarrow v_{1}<i\ \text{for every $i\in[1,m]$}\Big\}.\end{equation}
We draw the reader's attention to (\ref{EET}) and observe that one may write \begin{equation}\label{kala}\mathbb{E}\big(\lvert T_{n}(A)\rvert^{m}\big)\ll \sum_{\varphi\in \Phi}  E_{\varphi}, \end{equation} where upon writing every tuple $\bfx_{p}$ with $p\leq m$ as $\bfx_{p}=(x_{q,p})_{1\leq q\leq 2g+4}$ then
\begin{equation}\label{alat}E_{\varphi}=\sum_{\bfx_{1}\in T_{n}}\mathbb{P}(\text{Set}(\bfx_{1})\in A)\cdots\mathop{{\sum_{\substack{\bfx_{i}\in T_{n}\\ \text{Set}(\bfx_{i})=\bfr_{i}\cup \bfx_{i}'\\ \bfr_{i}\cap \bfx_{i}'=\o}}}^*}\mathbb{P}(\bfx_{i}'\in A) \cdots\mathop{{\sum_{\substack{\bfx_{m}\in T_{n}\\ \text{Set}(\bfx_{m})=\bfr_{m}\cup \bfx_{m}'\\ \bfr_{m}\cap \bfx_{m}'=\o}}}^*}\mathbb{P}(\bfx_{m}'\in A)\end{equation} 
where $\bfr_{i}=\{x_{v}:\ \ \ v\in\varphi(i)\}$ and $\bfx_{i}'\cap\text{Set}(\bfx_{j})=\o$ for every $j<i$, the set $\bfx_{i}'$ being uniquely determined by $\bfr_{i}$ and $\bfx_{i}$. We also introduce for $1\leq i\leq m$ with $\bfr_{i}\neq \o$ the set
$$\Gamma_{i}=\Big\{(\mathcal{A}_{i},\rho_{i}):\ \ \ \ \mathcal{A}_{i}\subset [1,2g+4],\ \ \ \lvert \mathcal{A}_{i}\rvert\geq \lvert\bfr_{i}\rvert,\ \ \ \rho_{i}:\mathcal{A}_{i}\rightarrow\bfr_{i}\ \ \text{surjective}\Big\}.$$ 

We then present 
\begin{equation}\label{YYY}\Lambda^{\varphi}=\Big\{  \Lambda=(\lambda_{i})_{i\leq m}:\ \ \lambda_{i}= (\mathcal{A}_{i},\rho_{i})\in  \Gamma_{i}\ \text{when $\bfr_{i}\neq \o$},\ \ \ \ \ \ \lambda_{i}=\varepsilon_{i}\in\{0,1\}\ \text{if $\bfr_{i}=\o$}        \Big\},\end{equation}
 and write each individual sum with $\bfr_{i}\neq \o$ as
$$\sum_{\substack{\bfx_{i}\in T_{n}\\ \text{Set}(\bfx_{i})=\bfr_{i}\cup \bfx_{i}'\\ \bfr_{i}\cap \bfx_{i}'=\o}}\mathbb{P}(\bfx_{i}'\in A) =\sum_{\substack{\mathcal{A}_{i}\subset [1,\ldots,2g+4],\\ \lvert \mathcal{A}_{i}\rvert\geq \lvert\bfr_{i}\rvert\\ \rho_{i}:\mathcal{A}_{i}\rightarrow\bfr_{i}\ \text{surj}}}\sum_{\substack{\bfx_{i}\in T_{n}\\ x_{a,i}=\rho_{i}(a)\\ a\in\mathcal{A}_{i}}}\mathbb{P}(\bfx_{i}'\in A).$$ When $\bfr_{i}=\o$ we write $\bfx_{i}\in T_{n}^{0}$ to denote $\bfx_{i}\in T_{n}^{dep}$ and $\bfx_{i}\in T_{n}^{1}$ to denote that $\bfx_{i}\in T_{n}^{ind}$. Equipped with the above notation we may express for each $\varphi\in \Phi$ the sum $E_{\varphi}$ as
\begin{equation}\label{RRRR}  E_{\varphi}=\sum_{\substack{\Lambda\in \Lambda^{\varphi}}}E_{\varphi,\Lambda}\end{equation}
with 
\begin{equation}\label{list}E_{\varphi,\Lambda}=\sum_{\substack{\bfx_{1}\in T_{n}^{\varepsilon_{1}}\\ \text{Set}(\bfx_{1})=\bfx_{1}'}}\mathbb{P}(\bfx_{1}'\in A)\cdots \sum_{\substack{\bfx_{i}\in T_{n}\\ x_{a,i}=\rho_{i}(a)\\ a\in\mathcal{A}_{i}}}\mathbb{P}(\bfx_{i}'\in A)\cdots \sum_{\substack{\bfx_{m}\in T_{n}\\ x_{a,m}=\rho_{m}(a)\\ a\in\mathcal{A}_{m}}}\mathbb{P}(\bfx_{m}'\in A),\end{equation}
wherein $\bfx_{i}'$ is defined as above. In order to make further progress in the argument we shall classify the above sums running over $\bfx_{i}$ according to the intersection of $\text{Set}(\bfx_{i})$ with elements of tuples in preceding sums. We first recall (\ref{alej}), (\ref{Tnre3}), (\ref{Tnr3n}), (\ref{NNN}) and (\ref{Defi}), and denote by $\mathcal{K}'$ to the collection of sums running over either:

$i)$ Tuples $\bfx_{i}\in T_{n}^{dep}$ and such that $\text{Set}(\bfx_{i})=\bfx_{i}'$.

$ii)$ Tuples $\bfx_{i}\in T^{l}_{n,\bfr_{i}}$ for some $l\in\{0,1,4\}$.

$iii)$ Tuples that are one of the components of the pair $(\bfx_{i},\bfx_{j})$ for $i<j$ with $\bfx_{i}\in T_{n}^{ind}$ and $\bfx_{j}$ not satisfying the properties in $ii)$, for which $\text{Set}(\bfx_{i})=\bfx_{i}'$, that also satisfy the relation $\text{Set}(\bfx_{i})\neq\text{Set}(\bfx_{i})\cap \bfr_{j}\neq\o$, with the property that if \begin{equation}\label{sole}\text{Set}(\bfx_{i})\cap \bfr_{j}=\{x_{1,i},x_{2,i},x_{3,i}\}=\{x_{1,j},x_{2,j},x_{3,j}\},\end{equation} wherein we wrote $\bfx_{i}=(x_{1,i},\ldots,x_{2g+4,i})$ and $\bfx_{j}=(x_{1,j},\ldots,x_{2g+4,j})$, then $$\text{Set}(x_{4,j},\ldots,x_{2g+4,j})\cap \bfr_{j}\neq \o,$$ and such that there is no $i<j_{1}<j$ with $(\bfx_{i},\bfx_{j_{1}})$ satisfying the earlier exposed properties. The reader may observe that these cases are encoded by $X_{n,\bfr}(A)$ in (\ref{tama}). 

$iv)$ Tuples that are one of the components of the pair $(\bfx_{i},\bfx_{j})$ for $i<j$ not satisfying the conditions described in $iii)$ for which $\bfx_{i}\in T_{n,\bfr_{i}}^{5}$ with $\bfr_{i}=\{x_{1,i},x_{2,i},x_{3,i}\}$ and $\bfx_{j}\in T_{n,\bfr_{j}}^{5}$ satisfying $\o\neq\bfx_{i}'\cap \bfr_{j}$ with either $\bfx_{i}'\cap \bfr_{j}\neq \text{Set}(x_{4,j},\ldots,x_{2g+4,j})$ or $\{x_{1,j},x_{2,j},x_{3,j}\}\cap \bfx_{i}'\neq\o$. As above, the sum over such tuples should also satisfy that there is no $i<j_{1}<j$ with $\bfx_{j_{1}}$ having the earlier exposed properties of $\bfx_{j}$. The reader may note that these instances are encoded by means of $Y_{n,\bfr}(A)$ in (\ref{Bruja}).

$v)$ Tuples that are one of the components of the pair $(\bfx_{i},\bfx_{j})$ with $i<j$ for which $\bfx_{i}\in T_{n,\bfr_{i}}^{5}$ with $g+2\leq\lvert \bfr_{i}\rvert\leq 2g+3$ such that $\bfx_{j}\in \mathcal{Z}_{n,\bfr_{j}\setminus\overline{\bfx},\overline{\bfx}}$ and the subset $\o\neq\overline{\bfx}\subset \bfx_{i}'\cap \bfr_{j},$ with no $i<j_{1}<j$ with $\bfx_{j_{1}}$ satisfying the earlier exposed properties of $\bfx_{j}$.

$vi)$ Tuples that are one of the components of $(\bfx_{i},\bfx_{j},\bfx_{k})$ for $i<j<k$ such that $\bfx_{h}\in T_{n,\bfr_{h}}^{5}$ with $g+2\leq\lvert \bfr_{h}\rvert\leq 2g+3$ for $h=i,j$ with $\bfx_{k}\in \mathcal{Z'}_{n,\bfr_{k}\setminus\overline{\bfx},\overline{\bfx}}$ and the subset $\overline{\bfx}\subset \mathbb{N}$ satisfies $\o\neq\overline{\bfx}\subset \text{Set}(\bfx_{i}',\bfx_{j}')\cap \bfr_{k}$. Moreover, we also impose the condition that there is no $(\bfx_{i_{1}},\bfx_{j_{1}},\bfx_{k_{1}})$ satisfying the above properties for some $i_{1},j_{1}<k_{1}$ with $\{i,j\}\cap\{i_{1},j_{1}\}\neq \o$, say $i=i_{1}$, and either $j_{1}<j$ or if $j=j_{1}$ then $k_{1}<k$. The preceding condition assures that for either fixed $i$ or $j$  then the other positions of the components of the triple are determined.

In view of Lemmata \ref{lem7}, \ref{lem7.1}, \ref{lem7.2} and Propositions \ref{propo7}, \ref{cila} and \ref{cilas} it transpires that the contribution of the sums running over tuples described in $i),ii),v)$ and $vi)$ shall be $O(n^{-\delta_{g}})$, the contribution of the ones corresponding to $iii)$ and $iv)$ being $O(n^{-\delta_{g}})$ as well when the sums over tuples $\bfx_{i}$ are translated suitably as we shall discuss later in the argument.

We shall introduce as well the collection of sums $\mathcal{K}$ which do not belong to $\mathcal{K}'$ that run over tuples $\bfx_{i}\in T_{n}^{ind}$ satisfying $\text{Set}(\bfx_{i})=\bfx_{i}'$. We also denote by $\mathcal{K}_{1}$ to the collection of sums which do not belong to $\mathcal{K}'$ running over:

$i)$ Tuples $\bfx_{i}$ for which there exists some $j<i$ such that $\bfx_{j}\in \mathcal{K}$ satisfies the relation $\{x_{1,i},x_{2,i},x_{3,i}\}=\{x_{1,j},x_{2,j},x_{3,j}\}\in R_{n}$ with $\bfx_{i}$ having the property that $$\text{Set}(x_{4,i},\ldots,x_{2g+4,i})=\bfx_{i}'.$$ Note that in the preceding lines $\bfx_{j}\in \mathcal{K}$ means that the corresponding sum running over such tuples is in $\mathcal{K}$, such an abuse of notation being often employed henceforth.

$ii)$ Tuples $\bfx_{i}$ with $\text{Set}(\bfx_{i})=\text{Set}(\bfx_{j})$ for some $j<i$ with either $\bfx_{j}$ being as in $i)$ in the above lines or $\bfx_{j}\in\mathcal{K}$.

Lastly we introduce the collection $\mathcal{K'}_{1}$ of sums not in $\mathcal{K}'$ and running over:

$i)$ Tuples $\bfx_{i} \in T_{n,\bfr_{i}}^{5}$ having the property that $$ \bfr_{i}\subset \bigcup_{k=1}^{k_{0}}\bfx_{j_{k}}'$$
with $\bfx_{j_{k}}\in \mathcal{K'}$ for $j_{k}<i$ and $3\leq k_{0}\leq 2g+4$.

$ii)$ Tuples $\bfx_{i}\in T_{n,\bfr_{i}}^{5}$ with $x_{k,i}\in\bfx_{j_{k}}'$ for $\bfx_{j_{k}}\in \mathcal{K'}$ and $j_{k}<i$ for $1\leq k\leq 3$, and satisfying $$\text{Set}(x_{4,i},\ldots,x_{2g+4,i})=\text{Set}(x_{4,l},\ldots,x_{2g+4,l})$$ for some $\bfx_{l}\in\mathcal{K}_{1}$ with $l<i$.

$iii)$ Tuples $\bfx_{i} \in T_{n,\bfr_{i}}^{5}$ with $x_{k,i}\in\bfx_{j_{k}}'$ for $\bfx_{j_{k}}\in \mathcal{K'}$ and $j_{k}<i$ for $1\leq k\leq 3$, and having the property that either $x_{2l+1}\in \bfx_{i_{l}}'$ or $x_{2l+2}\in\bfx_{i_{l}}'$ for $2\leq l\leq g+1$ with $\bfx_{i_{l}}\in\mathcal{K'}$ and $i_{l}<i$ with either $x_{4}\in \bfx_{i_{1}}'$ or $\{x_{2l_{0}+1},x_{2l_{0}+2}\}\subset\bfx_{i_{l_{0}}}'\cup\bfx_{i_{1}}'$ for some $i_{1}<i$ and some $2\leq l_{0}\leq g+1$.

\begin{lem}\label{leem}
The preceding list covers all possible cases for the sums over tuples $\bfx_{i}$ in (\ref{list}) according to the intersection of $\text{Set}(\bfx_{i})$ with elements of tuples in preceding sums.
\end{lem}
\begin{proof}
We may assume that $\bfx_{j}$ is not in $\mathcal{K}'$ or $\mathcal{K}$, and hence $\bfx_{j}\in T_{n,\bfr_{j}}^{5},$ since otherwise the remaining cases would belong to the aforementioned collections, with either $\lvert\bfr_{j}\rvert=3$ or $g+2\leq \lvert\bfr_{j}\rvert\leq 2g+4$ in view of the discussion after (\ref{Tnr3n}). In the first instance then $\{x_{1,j},x_{2,j},x_{3,j}\}= \bfr_{j},$ and we may suppose that one of the three elements, say $x_{1,j}$, has the property that $x_{1,j}\in\bfx_{i}'$ for some $i<j$ with $\bfx_{i}\notin \mathcal{K'}$ since otherwise that would correspond to the case $i)$ cognate to $\mathcal{K}_{1}'$. We note that if $\bfx_{i}\in T_{n,\bfr_{i}}^{5}$ for $g+2\leq \lvert \bfr_{i}\rvert\leq 2g+3$ then such an instance would be encompassed by $v)$ of $\mathcal{K'}$, since if there was $i<j_{1}<j$ with $(\bfx_{i},\bfx_{j_{1}})$ satisfying $v)$ then $\bfx_{i}\in\mathcal{K}'$, and if on the contrary $\lvert \bfr_{i}\rvert=3$ one would instead be in the case $iv)$ pertaining to $\mathcal{K'}$.  Since whenever $1\leq \lvert\bfr_{i}\rvert\leq 2$ one would have $\bfx_{i}\in T_{n,\bfr_{i}}^{l}$ for $l=0,1$ and hence $\bfx_{i}\in\mathcal{K}'$, it therefore would transpire that $\bfr_{i}=\o$. This would then entail in view of $iii)$ in (\ref{sole}) that $\text{Set}(\bfx_{i})\cap\text{Set}(\bfx_{j})=\{x_{1,i},x_{2,i},x_{3,i}\}=\{x_{1,j},x_{2,j},x_{3,j}\}$ with $\bfx_{j}'=\{x_{4,j},\ldots,x_{2g+4,j}\}$, such a situation being described in $i)$ of $\mathcal{K}_{1}$. 

When on the contrary $g+2\leq \lvert\bfr_{j}\rvert\leq 2g+4$ then we may assume that there is some $i<j$ with $\bfx_{i}\notin \mathcal{K'}$ and $\bfx_{i}'\cap\bfr_{j}\neq \o$ since otherwise one would be in the case $i)$ of $\mathcal{K}_{1}'$. The instance $x_{1,j}\in\bfx_{i}'\cap\bfr_{j}$ and similarly for $x_{2,j}$ and $x_{3,j}$, would follow by a similar argument as above save the instance $\bfr_{i}=\o$, in which case one would have $\text{Set}(\bfx_{i})=\text{Set}(\bfx_{j})$ as in $ii)$ of $\mathcal{K}_{1}$. We may therefore suppose that $x_{k,j}\in\bfx_{j_{k}}'$ for $1\leq k\leq 3$ and $j_{k}<j$ with $\bfx_{j}\in\mathcal{K}'$, and assume without loss of generality that $x_{2l,j}\in\bfx_{i}'\cap\bfr_{j}$ for some $2\leq l\leq g+2$ (and similarly for odd indexes) with $\bfx_{i}\notin \mathcal{K'}$. One may utilise the preceding argument to discard the instance $\bfr_{i}=\o$, the cases $1\leq \lvert\bfr_{i}\rvert\leq 2$ entailing $\bfx_{i}\in \mathcal{K'}$ which would contradict our hypothesis. If $\lvert \bfr_{i}\rvert=3$ then in view of $iv)$ after (\ref{sole}) and the remarks after (\ref{Tnr3n}), one would then have that $\text{Set}(x_{4,i},\ldots,x_{2g+4,i})\cap \text{Set}(\bfx_{j})= \text{Set}(x_{4,j},\ldots,x_{2g+4,j})$, such an instance being encompassed in $ii)$ cognate to $\mathcal{K}_{1}'$. When instead $g+2\leq\lvert \bfr_{i}\rvert\leq 2g+3$ for all such $l$ satisfying the aforementioned condition then one must necessarily have that $x_{2l-1,j}\in \bfr_{j}$ since otherwise $\bfx_{j}\in \mathcal{Z}_{n,\bfr_{j}\setminus\overline{\bfx},\overline{\bfx}}$ for some $\o\neq\overline{\bfx}\subset \bfx_{i}'\cap \bfr_{j}$ as in $v)$ pertaining to $\mathcal{K}'$. If $x_{2l_{0},j}\in\bfx_{i}'\cap\bfr_{j}$ and $x_{2l_{0}-1,j}\in \bfx_{k}'\cap\bfr_{j}$ for some $k<j$ with  $g+2\leq \lvert\bfr_{k}\rvert\leq 2g+3$ and $\bfx_{i},\bfx_{k}\notin \mathcal{K'}$ and some $2\leq l_{0}\leq g+2$, then $\bfx_{j}\in  \mathcal{Z'}_{n,\bfr_{j}\setminus\overline{\bfx},\overline{\bfx}}$, where the subset $\overline{\bfx}\subset \mathbb{N}$ satisfies $\o\neq\overline{\bfx}\subset \text{Set}(\bfx_{i}',\bfx_{k}')\cap\bfr_{j}$. The latter instance would correspond to the case $vi)$ of $\mathcal{K'}$. Otherwise, it transpires upon recalling the observation after (\ref{Tnr3n}) that either $x_{2l,j}\in\bfx_{i_{l}}'\cap\bfr_{j}$ or $x_{2l-1,j}\in\bfx_{i_{l}}'\cap\bfr_{j}$ for each $3\leq l\leq g+2$ and some $\bfx_{i_{l}}\in \mathcal{K'}$ and $i_{l}<j$, the corresponding tuple having the additional property that either $x_{4}\in \bfr_{j}$ or $\{x_{2l_{1}+1},x_{2l_{1}+2}\}\subset\bfr_{j}$ for some $2\leq l_{1}\leq g+1$, such an instance being described in $iii)$ of $\mathcal{K}_{1}'$. 
\end{proof}
Equipped with the above considerations we are prepared to prove the following statement.
\begin{prop}\label{prop3}
Whenever $\lambda=\lambda(g)$ is a sufficiently large constant one has with probability $1$ that for sufficiently large $n\in\mathbb{N}$ then
$$\lvert T_{n}(A)\rvert\ll \lambda^{2}\log n.$$ 
\end{prop}
\begin{proof}
We observe that as a consequence of Lemma \ref{leem} then \begin{equation}\label{kas}m=\lvert \mathcal{K}\rvert+\lvert\mathcal{K}_{1}\rvert+\lvert \mathcal{K'}\rvert+\lvert \mathcal{K}'_{1}\rvert.\end{equation} 
We further recall (\ref{phi}) and (\ref{YYY}) and denote for non-negative integers $K,K',K_{1},K_{1}'\leq m$ satisfying (\ref{kas}) and $\varphi\in\Phi$ by $\Lambda_{K,K',K_{1},K_{1}'}^{\varphi}$ to the subset of $\Lambda\in \Lambda^{\varphi}$ having the property that $K,K',K_{1}$ and $K_{1}'$ of the associated sums in (\ref{list}) are in $\mathcal{K},\mathcal{K}',\mathcal{K}_{1}$ and $\mathcal{K}_{1}'$ respectively. We then observe by recalling (\ref{RRRR}) that 
\begin{equation}\label{nene}\sum_{\varphi\in \Phi}  E_{\varphi}\ll \sum_{\substack{K,K',K_{1},K_{1}'\geq 0\\ K+K'+K_{1}+K_{1}'=m}}\sum_{\varphi\in\Phi}\ \ \ \sum_{\Lambda\in\Lambda_{K,K',K_{1},K_{1}'}^{\varphi}}E_{\varphi,\Lambda},\end{equation} and devote the upcoming lines to estimate \begin{equation*}\sum_{\varphi\in\Phi}\lvert \Lambda_{K,K',K_{1},K_{1}'}^{\varphi}\rvert.\end{equation*} We also denote by $\mathcal{S}_{q}$ for $1\leq q\leq 12$ to each of the $12$ collection of tuples described after (\ref{list}) and introduce $$\Upsilon_{m}=\{\upsilon: [1,m]\rightarrow [1,12]\ \  \}.$$ We then write for $\upsilon\in\Upsilon_{m}$ and $\varphi\in\Phi$ as above $\Lambda_{K,K',K_{1},K_{1}'}^{\varphi,\upsilon}$ to denote the subset of $\Lambda\in\Lambda_{K,K',K_{1},K_{1}'}^{\varphi}$ having the property that for each $i\leq m$ then the sum running over $\bfx_{i}$ in the expression (\ref{list}) associated to $\Lambda$ belongs to the collection $\mathcal{S}_{\upsilon(i)}.$ The preceding observation in conjunction with the fact that $\lvert \Lambda_{K,K',K_{1},K_{1}'}^{\varphi}\rvert=O(H_{g}^{m})$ for some constant $H_{g}>0$ and each $\varphi$ enables one to obtain \begin{equation}\label{HHH}\sum_{\varphi\in\Phi}\lvert \Lambda_{K,K',K_{1},K_{1}'}^{\varphi}\rvert\ll H_{g}^{m}\sum_{\substack{\varphi\in\Phi\\ \Lambda_{K,K',K_{1},K_{1}'}^{\varphi}\neq \o}}1\ll H_{g}^{m}\sum_{\upsilon\in\Upsilon_{m}}\sum_{\substack{\varphi\in\Phi\\ \Lambda_{K,K',K_{1},K_{1}'}^{\varphi,\upsilon}\neq \o}}1.\end{equation}
It then transpires after recalling the above definitions that in order to count for fixed $\upsilon$ the number of $\varphi$ satisfying the underlying condition in the sum on the right side it would suffice to count for each $1\leq q\leq 12$  and each $i_{q}\leq m$ such that $\upsilon(i_{q})=q$ the number of possible $\bfr_{i_{q}}$ as in (\ref{alat}) once the preceding tuples $\bfx_{1},\ldots,\bfx_{i_{q}-1}$ are fixed.

Indeed, whenever in the position $1<i\leq m$ the sum is running over tuples $\bfx_{i}$ as described in $i)$ of $\mathcal{K}_{1}$ then one would have \begin{equation*}\bfr_{i}=\{x_{1,j_{0}},x_{2,j_{0}},x_{3,j_{0}}\}\end{equation*} for some $1\leq j_{0}<i$, whence once the tuples $\bfx_{1},\ldots, \bfx_{i-1}$ are fixed, there are at most $i$ possibilities for $\bfr_{i}$. Likewise, for sums running over tuples $\bfx_{i}$ as described in $ii)$ of $\mathcal{K}_{1}$ then $\bfr_{i}$ is completely determined by some $1\leq j<i$, whence an analogous argument would yield the same conclusion, it entailing that there are at most $O(m^{K_{1}})$ possibilities for the sequence of sets $(\bfr_{i})_{i}$ corresponding to sums running over $\mathcal{K}_{1}$.

We next observe upon denoting by $\mathcal{K}_{2}$ to the collection of sums satisfying $i)$ in the description of $\mathcal{K}_{1}'$ and $K_{2}=\lvert \mathcal{K}_{2}\rvert$ that there are at most $O(K'^{2g+4})$ possibilities for each $\bfr_{i}$ with fixed $i\leq m$, whence the number of possibilities for the sequence of sets $(\bfr_{i})_{i}$ corresponding to sums in $\mathcal{K}_{2}$ would then be $O(C_{g}^{K_{2}}K'^{(2g+4)K_{2}})$ for some constant $C_{g}>0$. It also transpires when denoting $K_{3}$ to the number of sums satisfying $ii)$ in the description of $\mathcal{K'}_{1}$ that there are at most $O(mK'^{3})$ possible choices for each $\bfr_{i}$ and fixed $i$, whence the total number of such $\bfr_{i}$ would be bounded above by
$$C_{g}^{m}m^{K_{3}}K'^{3K_{3}}\ll C_{g}^{m}m^{K_{3}}K'^{3m}.$$

We shift our attention to the tuples $iii)$ in the description of $\mathcal{K}_{1}'$, denote by $K_{4}$ to the analogous cardinality and observe in view of the fact that tuples lie on $T_{n}$ that once the underlying tuples $\bfx_{i_{l}}$ and $\bfx_{i_{1}}$ are fixed then there are $O_{g}(1)$ choices for the tuple $\bfx_{i}$. Therefore, the number of $\bfr_{i}$ for fixed $i$ is bounded above by $O_{g}(mK'^{g+3})$, and hence the number of possibilities for the corresponding collection of $\bfr_{i}$ would be bounded above by
$$C_{g}^{K_{4}}m^{K_{4}}K'^{(g+3)K_{4}}\ll C_{g}^{m}m^{K_{4}}K'^{(g+3)m}.$$
We also note that for each sum running over tuples $\bfx_{i}$ in $\mathcal{K}'$ then once the tuples $\bfx_{j}$ with $j<i$ are fixed there are at most $\big((2g+4)i\big)^{2g+4}$ possibilities for $\bfr_{i}$. The above discussion reveals that there are $O\big((C_{g}m)^{(2g+4)K'}\big)$ possibilities for the sets $\bfr_{i}$ corresponding to sums in $\mathcal{K'}$. We then deduce upon noting $K+K_{1}+K_{3}+K_{4}\leq m$
that the preceding discussion in conjunction with (\ref{HHH}) and the fact that $\lvert \Upsilon_{m}\rvert=12^{m}$ delivers the estimate
\begin{align}\label{sanse}\sum_{\varphi\in\Phi}\lvert \Lambda_{K,K',K_{1},K_{1}'}^{\varphi}\rvert &\ll   C_{g}^{m}m^{K_{1}+K_{3}+K_{4}+(2g+4)K'}K'^{(3g+10)m}\nonumber
\\
& \ll C_{g}^{m}m^{m-K+(2g+4)K'}K'^{(3g+10)m}.\end{align}

We shift our attention to (\ref{list}) and (\ref{nene}) and note that in order to estimate $E_{\varphi,\Lambda}$ for fixed $(\varphi,\Lambda)$ it transpires in view of Propositions \ref{propo7} and \ref{cila} that moving a positive proportion of the sums over $\bfx_{i}$ to the position $j-1$ for pairs $(\bfx_{i},\bfx_{j})$ with $i+1<j$ described in $iii)$ and $iv)$ pertaining to $\mathcal{K}'$  is a desideratum. We first assume that the pair $(\bfx_{i},\bfx_{j})$ with $i+1<j$ is as in $iv)$, observe that for fixed $j$ there are at most $2g+4$ indexes $i_{1}<j$ for which $(\bfx_{i_{1}},\bfx_{j})$ satisfies the above and take the largest of such indexes. It then seems apparent in view of the definitions $ii)$ and $iv)$ of $\mathcal{K}'$ that if $\bfx_{i}'\cap\text{Set}(\bfx_{i+1})\neq \o$ then either $\bfx_{i+1}\in T_{n,\bfr_{i+1}}^{l}$ for $l\in\{0,1,4\}$ or \begin{equation}\label{kkka}\{x_{4,i},\ldots,x_{2g+4,i}\}=\{x_{4,i+1},\ldots,x_{2g+4,i+1}\}\end{equation} with $\bfx_{i+1}\in T_{n,\bfr_{i+1}}^{5}$ and $\text{Set}(\bfx_{i+1})=\bfr_{i+1}$. By iterating the above argument it transpires that either one moves the sum over $\bfx_{i}$ to the position $j-1$ or to the position $k-1$ for some $i< k<j$ with $\bfx_{k}\in T_{n,\bfr_{k}}^{l}$, the parameter $l$ as above and $\bfr_{k}\cap \bfx_{i}'\neq\o.$ We note that for a fixed position $k$ as above there are at most $2g+4$ positions $i$ satisfying such a property.

We next focus on the instance $(\bfx_{i},\bfx_{j})$ with $i+1<j$ as in $iii)$ of $\mathcal{K}'$, note that for fixed $j$ there are at most $2g+4$ indexes $i_{1}<j$ for which $(\bfx_{i_{1}},\bfx_{j})$ satisfies the above and take the largest of such indexes. If $\text{Set}(\bfx_{i})\cap \text{Set}(\bfx_{i+1})\neq \o$ then either $\bfx_{i+1}\in T_{n,\bfr_{i+1}}^{l}$ for $l\in\{0,1,4\}$, one has $\text{Set}(\bfx_{i})=\text{Set}(\bfx_{i+1})$ or $\bfr_{i+1}=\{x_{1,i},x_{2,i},x_{3,i}\}$ and the tuple $(\bfx_{i},\bfx_{i+1})$ is as in $i)$ cognate to $\mathcal{K}_{1}$. If $\bfr_{i+1}=\{x_{1,i},x_{2,i},x_{3,i}\}$ and either $\bfx_{i+1}'\cap \text{Set}(\bfx_{l})=\o$ or $\text{Set}(\bfx_{l})=\bfr_{l}$ as in (\ref{kkka}) for all $i+1<l\leq j$ then one may move the sum over $\bfx_{i+1}$ to the position $j$ and iterate the process. If not the translation of the sum over $\bfx_{i+1}$ as described above would transport it to a position $p< m$ for which either $(\bfx_{i+1},\bfx_{p+1})$ is as in $iv)$ after (\ref{sole}) or $\bfx_{p+1}\in T_{n,\bfr_{p+1}}^{l}$ with $\bfx_{i+1}'\cap\bfr_{p+1}\neq\o$. If on the contrary $\text{Set}(\bfx_{i})\cap \text{Set}(\bfx_{i+1})= \o$ then one may flip the sums and iterate the above process. It would follow that either one would transport the sum over $\bfx_{i}$ to the position $j-1$ or to the position $k-1$ for some $i< k<j$ with $\bfx_{k}\in T_{n,\bfr_{k}}^{l}$, the number $l$ as above and $\bfr_{k}\cap \bfx_{i}'\neq\o,$ or one would transport for some $i<h<k_{1}\leq j$ the sum over tuples $\bfx_{h}$ with $\bfr_{h}=\{x_{1,i},x_{2,i},x_{3,i}\}$ to the position $k_{1}-1$ for some $i< k_{1}\leq j$ with either $(\bfx_{h},\bfx_{k_{1}})$ satisfying $iv)$ after (\ref{sole}) or $\bfx_{k_{1}}\in T_{n,\bfr_{k_{1}}}^{l}$, the number $l$ as above and $\bfr_{k_{1}}\cap \bfx_{h}'\neq\o.$ It then seems worth observing that for each such position $k_{1}$ there are at most $2g+4$ positions $h$ satisfying the above conditions, the position $i$ being uniquely determined by $h$. After changing the orders of summation appropiately in (\ref{alat}) as described above and denoting by $S_{c}$ and $S_{nc}$ to the number of consecutive and non-consecutive pairs of sums respectively running over tuples described in $iii)$ and $iv)$ in $\mathcal{K}'$, and $T_{2}$ to the number of sums running over tuples described in $ii)$ of $\mathcal{K'}$, it will transpire by the preceding discussion that there exists a constant $D_{g}>0$ such that 
\begin{equation}\label{alfs}S_{nc}\leq D_{g}(T_{2}+S_{c}).\end{equation}

 Before making further progress in the proof it seems worth noting that an application of Lemmata \ref{lem7.1} and \ref{lem7.2} reveals that whenever $\bfr_{j}\neq \o$ then \begin{equation}\label{fran}\sum_{\substack{\bfx_{j}\in T_{n}\\ \text{Set}(\bfx_{j})=\bfr_{j}\cup \bfx_{j}'\\ \bfr_{j}\cap \bfx_{j}'=\o}}\mathbb{P}(\bfx_{j}'\in A)\ll 1.\end{equation}
We resume our analysis by assuming that $(\bfx_{i},\bfx_{j})$ for $i+1<j$ is as in $v)$ a few paragraphs after (\ref{sole}) and justify why it is not necessary in this instance to translate the sums as was previously done. We take for fixed $j$ the largest of such $i$ as above and observe whenever $g+2\leq \lvert \bfr_{i}\rvert\leq 2g+3$ that
\begin{align*}\sum_{\substack{\bfx_{i}\in T_{n,\bfr_{i}}^{5}\\ \text{Set}(\bfx_{i})=\bfr_{i}\cup \bfx_{i}'\\ \bfr_{i}\cap \bfx_{i}'=\o}}&\mathbb{P}(\bfx_{i}'\in A) \cdots\sum_{\o\neq\overline{\bfx}\subset \bfx_{i}'\cap \bfr_{j}}\sum_{\substack{\bfx_{j}\in \mathcal{Z}_{n,\bfr_{j}\setminus \overline{\bfx},\overline{\bfx}}\\ \text{Set}(\bfx_{j})=\bfr_{j}\cup \bfx_{j}'\\ \bfr_{j}\cap \bfx_{j}'=\o}}\mathbb{P}(\bfx_{j}'\in A)
\\
\ll&\sum_{\substack{\bfx_{i}\in \mathcal{W}_{1,\bfr_{i}}\\ \text{Set}(\bfx_{i})=\bfr_{i}\cup \bfx_{i}'\\ \bfr_{i}\cap \bfx_{i}'=\o}}\mathbb{P}(\bfx_{i}'\subset A)
\cdots\sum_{\substack{\bfx_{j}\in T_{n,\bfr_{j}}^{5}\\ \text{Set}(\bfx_{j})=\bfr_{j}\cup \bfx_{j}'\\ \bfr_{j}\cap \bfx_{j}'=\o}}\mathbb{P}(\bfx_{j}'\in A)
\\
&+\sum_{\substack{\bfx_{i}\in \mathcal{W}_{2,\bfr_{i}}\\ \text{Set}(\bfx_{i})=\bfr_{i}\cup \bfx_{i}'\\ \bfr_{i}\cap \bfx_{i}'=\o}}\mathbb{P}(\bfx_{i}'\subset A)\cdots \sum_{\o\neq \overline{\bfx}\subset \bfx'_{i}\cap \bfr_{j}}\mathbb{E}\big( \lvert \mathcal{Z}_{n,\bfr_{j}\setminus \overline{\bfx},\overline{\bfx}}(A)\rvert \big|\bfr_{j},\overline{\bfx}\subset A\big),
\end{align*}
the sets $\mathcal{W}_{1,\bfr_{i}},\mathcal{W}_{2,\bfr_{i}}$ stemming from Lemma \ref{cilas}, and the dots expressing that there may possibly be other sums in between. Then, the application of such a lemma thereby entails that the inner sum in the second summand of the above equation is $O(n^{-\delta_{g}})$. For the first one, one would instead deduce by means of (\ref{fran}) that the innermost sum is $O(1)$ and would use Lemma \ref{cilas} again to derive that the first sum over $\bfx_{i}\in \mathcal{W}_{1,\bfr_{i}}$ would then be $O(n^{-\delta_{g}})$, as desired. If $(\bfx_{i},\bfx_{j},\bfx_{k})$ for $i<j<k$ is as in $vi)$ then for fixed $k$ there are at most $O_{g}(1)$ pairs $(i,j)$ satisfying the above, whence by taking the largest one in the lexicographical order, Lemma \ref{cilas} permits one to derive as above that either the inner sum over $\bfx_{j}$ or its counterpart over $\bfx_{i}$ is $O(n^{-\delta_{g}})$, the corresponding one over $\bfx_{k}$ being $O(1)$ after an application of (\ref{fran}).

Then, upon denoting by $T_{5}$ to the number of sums running over tuples in $i)$, $v)$ and $vi)$ of $\mathcal{K}'$, we observe that by the preceding discussion and (\ref{alfs}) there exists a constant $D_{g}'>0$ such that 
\begin{equation*} K'\leq  D_{g}'(T_{2}+S_{c}+T_{5}),\end{equation*}whence it transpires that
\begin{equation*}n^{-\delta_{g}(T_{2}+S_{c}+T_{5})}\ll n^{-\delta_{g}'K'}.\end{equation*}
We thereby deduce for given $\varphi\in\Phi$ and $\Lambda\in\Lambda_{K,K',K_{1},K_{1}'}^{\varphi}$ via the preceding discussion and Lemmata \ref{lem7}, \ref{lem7.1} and \ref{lem7.2} to bound the remaining sums in (\ref{list}) that correspond to the cases $i)$ and $ii)$ of $\mathcal{K'}$, Lemma \ref{lem7.2} to estimate the sums in $\mathcal{K}_{1}$ and $\mathcal{K}_{1}'$ and Lemma \ref{lem7} to estimate the sums in $\mathcal{K}$ that
$$E_{\varphi,\Lambda}\ll \Delta_{g}^{m}n^{-\delta_{g}(T_{2}+S_{c}+T_{5})}(\lambda \log n)^{K}\ll \Delta_{g}^{m}n^{-\delta_{g}'K'}(\lambda \log n)^{K}$$ with $\Delta_{g},\delta_{g}'>0$ fixed. Such a bound in conjunction with (\ref{kala}), (\ref{nene}) and (\ref{sanse}) enables one to derive
\begin{align}\label{RRR}\mathbb{E}\big(\lvert T_{n}(A)\rvert^{m}\big)&\ll C_{g}^{m}\lambda^{K}(\log n)^{m+4+(2g+4)K'}K'^{(3g+10)m}n^{-\delta_{g}'K'}\nonumber
\\
&\ll C_{g}^{m}\lambda^{K}(\log n)^{m+4}K'^{(3g+10)m}n^{-\delta_{g,0}K'}\end{align} for some fixed $C_{g},\delta_{g,0}>0$. It seems pertinent to note that $$K'^{(3g+10)m}n^{-\delta_{g,0}K'}\ll c_{g}^{m},$$ where $c_{g}>0$ is a constant, since indeed for some $c_{g}=C_{g}(C)>0$ sufficiently large one has $$(3g+10)\log K'\leq \log c_{g}+\delta_{0,g}K'.$$  

We then observe that an application of Markov's inequality yields
\begin{equation*}\mathbb{P}\big(\lvert T_{n}(A)\rvert>\lambda^{2}\log n\big)=\mathbb{P}\big(\lvert T_{n}(A)\rvert^{m}>\lambda^{2m}(\log n)^{m}\big)\leq \frac{\mathbb{E}\big(\lvert T_{n}(A)\rvert^{m}\big)}{\lambda^{2m}(\log n)^{m}},\end{equation*}
and combine the preceding estimates with (\ref{RRR}) to deduce that
\begin{equation*}\label{mark}\mathbb{P}\big(\lvert T_{n}(A)\rvert>\lambda^{2}\log n\big)\ll C_{g}^{m}\lambda^{K-2m}(\log n)^{4}\ll (C_{g}\lambda^{-1})^{m}(\log n)^{4}\ll n^{-2}\end{equation*} for sufficiently large $\lambda$. An application of Lemma \ref{prop2} (Borel-Cantelli) then delivers the desired result.
\end{proof}
 
\section{A lower bound for the counting function}\label{sec8}
We shall devote the remainder of the memoir to employ the preceding sequel of propositions for the purpose of concluding the proof of Theorem \ref{thm1.1}. To such an end we apologise for the abuse of notation and define for $N\in\mathbb{N}$ the random variable \begin{equation}\label{cont}A(N)=\sum_{n\leq N}\leavevmode\hbox{$1\!\rm I$}_{A}(n).\end{equation} In view of the considerations after (\ref{probas}) it is apparent that the above expression is a sum of independent random variables taking values $0$ and $1$ and satisfying $$\mathbb{E}(A(N))\asymp\lambda^{-1}\sum_{n\leq N}n^{-\alpha}\asymp \lambda^{-1}N^{1-\alpha}.$$ One then may apply Chernoff's inequality \cite[Corollary A-1.14]{Alon} to deduce for every $N\in\mathbb{N}$ and fixed $\delta>0$ the existence of a constant $c_{\delta}>0$ such that 
$$\mathbb{P}\big(\lvert A(N)-\mathbb{E}(A(N))\rvert\geq \delta\mathbb{E}(A(N))\big)\leq e^{-c_{\delta}N^{1-\alpha}}.$$
Therefore, an application of Lemma \ref{prop2} (Borel-Cantelli) enables one to deduce with probability $1$ that \begin{equation}\label{Am}A(N)\asymp \lambda^{-1} N^{1-\alpha}.\end{equation}

\emph{Proof of Theorem \ref{thm1.1}.} We apply Lemma \ref{lem1.4} and Proposition \ref{prop3} to derive that \begin{equation}\label{Raaa}R_{n}(A)\gg \lambda^{2g+5}\log n,\ \ \ \ \ \ \ \ \ \ \ \ \ \ \ \ \ \ \ \ \ \ \ \ \lvert T_{n}(A)\rvert\ll \lambda^{2} \log n\end{equation} with probability $1$. We take any such $A$ satisfying the property (\ref{Raaa}) and draw the reader's attention to equation (\ref{AAA}) to the end of observing upon recalling the definition of $A'$ right after (\ref{AAAAA}) and taking a sufficiently large but fixed constant $\lambda=\lambda(g)$ that there exists a $B_{2}[g]$ sequence $A'$ having the property that \begin{equation}\label{lower}R_{n}(A')\gg \lambda^{2g+5}\log n.\end{equation} 
It seems worth noting after a succint inspection of such a definition and upon considering for $N\in\mathbb{N}$ the counting function $A'(N)=\lvert A'\cap [1,N]\rvert$ that \begin{equation}\label{ine}A'(N)\leq A(N).\end{equation} 

The reader may observe that the same notation for the sake of concission has been utilised in (\ref{cont}) to denote the corresponding associated random variable. We thus sum over $n$ the equation (\ref{lower}) and obtain
$$N\log N\ll_{\lambda} \sum_{N/2\leq n\leq N}R_{n}(A')\ll  A'(N)^{2} A'(g_{\lambda}(N)).$$ We further remark that the bounds (\ref{Am}) and (\ref{ine}) permit one to deduce $$A'(g_{\lambda}(N))\leq A(g_{\lambda}(N))\ll_{\lambda}N^{\varepsilon(1-\alpha)}\log N ,$$ whence the preceding lines enable one to derive 
$$N\log N\ll_{\lambda}    A'(N)^{2} 
N^{\varepsilon(1-\alpha)}\log N.$$ Consequently, it transpires by simplifying accordingly the above line that
$$N^{1-\alpha}\ll_{\lambda} A'(N),$$ wherein the reader may find it useful to note upon recalling (\ref{alpha}) that $$1-\varepsilon(1-\alpha)=\frac{2g}{2g+1}=2(1-\alpha).$$
By the preceding discussion we conclude that the sequence $A'$ is a $B_{2}[g]$ sequence having the property (\ref{refi}) and such that its counting function satisfies the previous lower bound.

\end{document}